\documentclass[a4paper,10pt]{amsart}
\usepackage{amsmath}
\usepackage{amsthm}
\usepackage{amssymb}
\usepackage{hyperref}

\usepackage{mathrsfs}
\usepackage[all]{xy}
\usepackage{color}

\numberwithin{equation}{section}

\newcommand{\SDR}[5]{\xymatrix{*[r]{#1} \ar@<1ex>[r]^-{#3} \ar@(ul,dl)[]_{#5} & #2 \ar@<1ex>[l]^-{#4}}}
\newcommand{\bigSDR}[5]{\xymatrix{*[r]{#1} \ar@<1ex>[rr]^-{#3} \ar@(ul,dl)[]_{#5} && #2 \ar@<1ex>[ll]^-{#4}}}
\newcommand{\bigbigSDR}[5]{\xymatrix{*[r]{#1} \ar@<1ex>[rrr]^-{#3} \ar@(ul,dl)[]_{#5} &&& #2 \ar@<1ex>[lll]^-{#4}}}

\newcommand{\adjunction}[4]{\xymatrix{ {#1} \ar@<1ex>[r]^-{#3} & {#2} \ar@<1ex>[l]^-{#4}}}
\newcommand{\bigadjunction}[4]{\xymatrix{ #1 \ar@<1ex>[rr]^-{#3} && #2 \ar@<1ex>[ll]^-{#4}}}
\newcommand{\hugeadjunction}[4]{\xymatrix{ #1 \ar@<1ex>[rrr]^-{#3} &&& #2 \ar@<1ex>[lll]^-{#4}}}
\newcommand{\varadjunction}[4]{\xymatrix{ #3 \colon #1 \ar@<1ex>[r] & #2 \colon #4 \ar@<1ex>[l]}}

\newcommand{\tensor}{\otimes}
\newcommand{\cotensor}{\square}
\newcommand{\Hom}{\operatorname{Hom}}
\newcommand{\Map}{\operatorname{Map}}

\newcommand{\Coinv}{\operatorname{Coinv}}
\newcommand{\Ch}{\mathsf{Ch}}
\newcommand{\coker}{\operatorname{coker}}

\newcommand{\Ho}{\operatorname{Ho}}

\newcommand{\Coring}{\mathsf{Coring}}

\newcommand{\Alg}{\mathsf{Alg}}
\newcommand{\Coalg}{\mathsf{Coalg}}

\newcommand{\Comodalg}{\mathsf{ComodAlg}}

\newcommand{\Pair}{\mathsf{Triple}}

\newcommand{\kk}{\Bbbk}

\newcommand{\CC}{\mathscr{C}}
\newcommand{\DD}{\mathscr{D}}

\newcommand{\VV}{\mathscr{V}}
\newcommand{\UU}{\mathcal{U}}

\newcommand{\cat}{\mathscr}

\newcommand{\coequalizer}[5]{\xymatrix{{#1} \ar@<0.5ex>[r]^-{#4} \ar@<-0.5ex>[r]_-{#5} & {#2} \ar[r] & {#3}}}
\newcommand{\equalizer}[5]{\xymatrix{{#1} \ar[r] & {#2} \ar@<0.5ex>[r]^-{#4} \ar@<-0.5ex>[r]_{#5} & {#3}}}

\newcommand{\Desc}{\operatorname{Desc}}
\newcommand{\Hopf}{\operatorname{Hopf}}
\newcommand{\Gal}{\operatorname{Gal}}
\newcommand{\Cof}{\operatorname{Cof}}
\newcommand{\Can}{\operatorname{Can}}
\newcommand{\Prim}{\operatorname{Prim}}
\renewcommand{\Bar}{\operatorname{Bar}}

\newcommand{\vp}{\varphi}
\newcommand{\ve}{\varepsilon}
\newcommand{\Om}{\Omega}

\newcommand{\Id}{\operatorname{Id}}
\newcommand\si{s^{-1}}
\newcommand\tot{\,\widetilde\otimes\,}
\newcommand\G {\mathbb G}

\newtheorem{theorem}{Theorem}[section]
\newtheorem{proposition}[theorem]{Proposition}

\newtheorem{corollary}[theorem]{Corollary}
\newtheorem{lemma}[theorem]{Lemma}

\theoremstyle{definition}
\newtheorem{definition}[theorem]{Definition}

\newtheorem{example}[theorem]{Example}
\newtheorem{examples}[theorem]{Examples}
\newtheorem{remark}[theorem]{Remark}

\newtheorem{convention}[theorem]{Hypothesis}
\newtheorem{notation}[theorem]{Notation}

\numberwithin{theorem}{section}
\title{Homotopic Hopf-Galois extensions revisited}
\author{Alexander Berglund and Kathryn Hess}
\address{Department of Mathematics, Stockholm University, SE-106 91 Stockholm, Sweden}
\email{alexb@math.su.se}
\address{MATHGEOM, \'Ecole Polytechnique F\'ed\'erale de Lausanne, CH-1015 Lausanne, Switzerland}
\email{kathryn.hess@epfl.ch}

\thanks{This material is based upon work partially supported by the National Science Foundation under Grant No.~0932078000 while the second author was in residence at the Mathematical Sciences Research Institute in Berkeley, California, during the Spring 2014 semester. The second author was also supported during this project by the Swiss National Science Foundation, Grant No. 200020\_144393.}

 \keywords {Hopf-Galois extension, descent, Morita theory, model category.} 
 \subjclass [2010] {Primary: 16T05; Secondary: 13B05, 16D90, 16T15, 18G35, 18G55, 55U35}
\begin{document}

\begin{abstract} In this article we revisit the theory of homotopic Hopf-Galois extensions introduced in \cite{hess:hhg}, in light of the homotopical Morita theory of comodules established in \cite{berglund-hess:morita}. We generalize the theory to a relative framework, which we believe is new even in the classical context and which is essential for treating the Hopf-Galois correspondence in \cite{hess-karpova}. We study in detail homotopic Hopf-Galois extensions of differential graded algebras over a commutative ring, for which we establish a descent-type characterization analogous to the one Rognes provided in the context of ring spectra \cite{rognes}. An interesting feature in the differential graded setting is the close relationship between homotopic Hopf-Galois theory and Koszul duality theory.  We show that nice enough principal fibrations of simplicial sets give rise to homotopic Hopf-Galois extensions in the differential graded setting, for which this Koszul duality has a familiar form.
\end{abstract}

\maketitle

\tableofcontents

\section{Introduction}
The theory of Hopf-Galois extensions of associative rings, introduced by Chase and Sweedler \cite{chase-sweedler} and by Kreimer and Takeuchi \cite{kreimer-takeuchi}, generalizes Galois theory of fields, replacing the action of a group by the coaction of a Hopf algebra. Inspired by Rognes' theory of Hopf-Galois extensions of ring spectra \cite{rognes}, the second author laid the foundations for  a theory of homotopic Hopf-Galois extensions in an arbitrary monoidal model category in \cite{hess:hhg}, but the necessary model category structures were not well enough understood to make it possible to compute many examples.  Since then, considerable progress has been made in elaborating these model category structures (e.g., \cite{bhkkrs}, \cite{hess-shipley}, \cite{hkrs}), so that the time is ripe to revisit this subject. 

In this article we develop anew the theory of homotopic Hopf-Galois extensions, in light of the homotopical Morita theory of comodules established in \cite{berglund-hess:morita}. Moreover we generalize the theory to a relative framework, which we believe is new even in the classical context and which is  essential for treating the Hopf-Galois correspondence in \cite{hess-karpova}. We also provide a descent-type characterization of homotopic Hopf-Galois extensions of finite-type differential graded algebras over a field, analogous to \cite[Proposition 12.1.8]{rognes}.

\subsection{The classical framework}

Classical Hopf-Galois extensions show up in a wide variety of mathematical contexts. For example, faithfully flat HG-extensions over the coordinate ring of an affine group scheme $G$ correspond to $G$-torsors.  By analogy, if a Hopf algebra $H$ is the coordinate ring of a quantum group, then an $H$-Hopf-Galois extension can be viewed as a noncommutative torsor with  the quantum group as its structure group.  It can moreover be fruitful to  study Hopf algebras via their associated Hopf-Galois extensions, just as algebras are studied via their associated  modules. 

For an excellent introduction to the classical theory of Hopf-Galois extensions, we refer the reader to the survey articles by Montgomery \cite{montgomery} and Schauenburg \cite{schauenburg}.  We recall here only the definition and two elementary examples, which can be found in either of these articles.

\begin{definition}\label{defn:hg-classical}  Let $R $ be a commutative ring, and let $H$ be a $R $-bialgebra. Let $\varphi\colon   A\to B$ be a  homomorphism of right $H$-comodule algebras, where the $H$-coaction on $A$ is trivial.

The homomorphism $\varphi$ is an \emph{$H$-Hopf-Galois extension} if  
\begin{enumerate}
\item the composite
$$B\otimes_{A} B \xrightarrow {B\otimes _{A} \rho} B\otimes_{A}B\otimes H \xrightarrow{\mu \otimes H} B\otimes H,$$
where $\rho$ denotes the $H$-coaction on $B$, and $\mu$ denotes the multiplication map of $B$ as an $A$-algebra,
and 
\item the induced map
$$A\to B^{\mathrm{co}\, H} :=\{ b\in B\mid \rho(b)=b\otimes 1\}$$
\end{enumerate}
are both isomorphisms.
\end{definition}

\begin{notation}  The composite in (1), often denoted $\beta_{\vp}\colon   B\otimes_{A}B \to B\otimes H$, is called the \emph{Galois map}.
\end{notation}

\begin{examples}\begin{enumerate}
\item \cite[Example 2.3]{montgomery} Let $\kk \subset E$ be a field extension.  Let $G$ be a finite group that acts on $E$ through $\kk$-automorphisms, which implies that its dual $\kk ^{G}=\operatorname{Hom} (\kk [G], \kk)$ coacts on $E$.  The extension $E^{G}\subset E$ is $G$-Galois if and only if it is a $\kk ^{G}$-Hopf-Galois extension.
\item \cite[Theorem 2.2.7]{schauenburg}    Let $R $ be a commutative ring, $H$ a bialgebra over $R $ that is flat as $R $-module, and $A$ a flat $R $-algebra.  The trivial extension $A\to A\otimes H\colon   a \mapsto a\otimes 1$ is then an $H$-Hopf-Galois extension if $A\otimes H$ admits a \emph{cleaving}, i.e., a convolution-invertible morphism of $H$-comodules $H\to A\otimes H$.  In particular, the unit map $\kk \to H$ is an $H$-Hopf-Galois extension if and only if $H$ is a Hopf algebra.
\end{enumerate}
\end{examples}

\subsection{The homotopic framework}

In his monograph on Galois extensions of structured ring spectra \cite{rognes}, Rognes formulated a reasonable, natural definition of homotopic Hopf-Galois extensions of commutative ring spectra. Let $\vp\colon  A \to B$ be a morphism of commutative ring spectra, and let $H$ be a commutative ring spectrum equipped with a comultiplication $H\to H\wedge H$ that is a map of ring spectra, where $-\wedge -$ denotes the smash product of spectra.  Suppose that $H$ coacts on $B$ so that $\vp$ is a morphism of $H$-comodules when $A$ is endowed with the trivial $H$-coaction.  If the Galois map
$\beta_{\vp}\colon   B\wedge_{A}B \to B\wedge H$ (defined as above) and the natural map from $A$ to (an appropriately defined model of) the homotopy coinvariants of the $H$-coaction on $B$ are both weak equivalences, then $\vp\colon   A\to B$ is a homotopic $H$-Hopf-Galois extension in the sense of Rognes.

The unit map $\eta$ from the sphere spectrum $S$ to the complex cobordism spectrum $MU$ is an $S[BU]$-Hopf-Galois extension in this homotopic  sense. The diagonal $\Delta\colon  BU\to BU\times BU$ induces the comultiplication $S[BU]\to S[BU]\wedge S[BU]$, the Thom diagonal $ MU\to MU\wedge BU_{+}$ gives rise to the coaction of $S[BU]$ on $MU$, and $\beta_{\eta}\colon  MU\wedge MU \xrightarrow \sim  MU\wedge S[BU]$ is the Thom equivalence.

In  \cite[Proposition 12.1.8]{rognes}, Rognes provided a descent-type characterization of homotopic Hopf-Galois extensions.  Let $A^{\wedge}_{B}$ denote Carlsson's derived completion of $A$ along $B$ \cite{carlsson}. Rognes proved that if $\vp\colon  A\to B$ is such that $\beta_{\vp}$ is a weak equivalence, then it is a homotopic $H$-Hopf-Galois extension if and only if the natural map $A\to A^{\wedge}_{B}$ is a weak equivalence, which  holds if, for example, $B$ is faithful and dualizable over $A$ \cite[Lemma 8.2.4]{rognes}.

\subsection{Structure of this paper}

We begin in Section \ref{sec:morita} by summarizing from \cite{berglund-hess:morita}   those elements of the homotopical Morita theory of modules and comodules in a monoidal model category that are necessary in this paper.  In particular we recall conditions under which a morphism of corings induces a Quillen equivalence of the associated comodule categories (Corollary \ref{cor:QE conditions}).  

In Section \ref{sec:rel H-G}  we introduce a new theory of relative Hopf-Galois extensions, insisting on the global categorical picture.  We first treat the classical case, then introduce the homotopic version, providing relatively simple conditions under which a morphism of comodule algebras in a monoidal model category is a relative homotopic Hopf-Galois extension (Proposition \ref{prop:hhg conditions}).

We furnish a concrete illustration of the theory  of relative homotopic Hopf-Galois extensions in Section \ref{sec:ch cx}, where we consider the monoidal model category $\Ch_{R}$  of unbounded chain complexes over a commutative ring $R$.  After recalling from \cite{berglund-hess:morita}   the homotopy theory of modules and comodules in this case, we elaborate the homotopy theory of comodule algebras in $\Ch_{R}$, recalling the necessary existence result for model category structures from \cite{hkrs}, then describing and studying a particularly useful fibrant replacement functor, given by the cobar construction (Theorem \ref{thm:fib-repl}). Finally, we describe in detail the theory of relative homotopic Hopf-Galois extensions of differential graded algebras over a commutative ring $R$.  In particular we  establish the existence of a useful family of relative homotopic Hopf-Galois extensions analogous to the classical normal extensions (Proposition \ref{prop:hyp2}). We apply this family to proving, under reasonable hypotheses, that a morphism of comodule algebras is a relative homotopic Hopf-Galois extension if and only if it satisfies effective homotopic descent (Proposition \ref{prop:rognes-analogue}), a result analogous to \cite[Proposition 12.1.8]{rognes} for commutative ring spectra. As a consequence we establish an intriguing relationship between Hopf-Galois extensions and Koszul duality, implying in particular that, under reasonable hypotheses, if $A\to B$ is a homotopic Hopf-Galois extension with respect to some Hopf algebra $H$, where $B$ is contractible, then $H$ is Koszul dual to $A$ (Proposition \ref{prop:koszul}).  Finally, we explain how to associate a homotopic Hopf-Galois extension in differential graded setting naturally to a nice enough principal fibration of simplicial sets (Proposition \ref{prop:prin-fib}) and show that Koszul duality has a familiar form in this case (Remark \ref{rmk:Koszul-prinfib}).

\subsection{Conventions}
\begin{itemize}
\item All forgetful functors are denoted $\UU$.
\item Let $\adjunction \CC \DD LR$ be an adjoint pair of functors.  If $\CC$ is endowed with a model category structure, and  $\DD$ admits a model category structure for which the fibrations and weak equivalences are created in $\CC$, i.e., a morphism in $\DD$ is a fibration (respectively, weak equivalence) if and only if its image under $R$ is a fibration (respectively, weak equivalence) in $\CC$, then we say that it is \emph{right-induced} by the functor $R$.  Dually, if $\DD$ is endowed with a model category structure, and $\CC$ admits a model category structure for which the cofibrations and weak equivalences are created in $\DD$, i.e., a morphism in $\CC$ is a cofibration (respectively, weak equivalence) if and only if its image under $L$ is a cofibration (respectively, weak equivalence) in $\DD$, then we say that it is \emph{left-induced} by the functor $L$.
\end{itemize}

\section{Elements of homotopical Morita theory}\label{sec:morita}
In this section we recall from \cite {berglund-hess:morita} those elements of homotopical Morita theory for modules and comodules that are necessary for our study of homotopic Hopf-Galois extensions in monoidal model categories. Since the definitions and results in \cite{berglund-hess:morita} are couched in a more general framework than we need in this article, we specialize somewhat here, for the reader's convenience.

\subsection{Homotopy theory of modules}

Let $(\VV,\tensor,\kk)$ be a monoidal category. Let $\Alg_{\VV}$ denote the category of algebras in $\VV$, i.e., of  objects $A$ in $\VV$ together with two maps $\mu\colon A\tensor A\rightarrow A$ and $\eta\colon \kk\rightarrow A$ that satisfy the usual associativity and unit axioms. Dually, the category of \emph{coalgebras} in $\VV$, i.e., objects in $\VV$ that are endowed with a coassociative comultiplication and a counit, is denoted $\Coalg_{\VV}$.

A right (respectively, left) module over  an algebra $A$ is an object $M$ in $\VV$ together with a map $\rho\colon M\tensor A\rightarrow M$ (respectively, $\lambda \colon A\otimes M \to M$) satisfying the usual axioms for an action. We let $\VV_A$ (respectively, ${}_{A}\VV$) denote the category of right  (respectively, left) $A$-modules in $\VV$.  {We usually omit the multiplication and unit from the notation for an algebra and the action map from the notation for an $A$-module.}

Schwede and Shipley established reasonable conditions, satisfied by many model categories of interest, under which module categories inherit a model category structure from the underlying category.

\begin{theorem}\cite[Theorem 4.1]{schwede-shipley} \label{theorem:module model}
Let $\VV$ be a symmetric monoidal model category. If $\VV$ is cofibrantly generated and satisfies the monoid axiom, and every object of $\VV$ is small relative to the whole category, then the category $\VV_A$ of right $A$-modules admits a model structure that is right induced from the adjunction
$$\bigadjunction{\VV}{\VV_A}{-\tensor A}{\UU},$$
and similarly for the category ${}_{A}\VV$ of left $A$-modules.
\end{theorem}

Categories of $A$-modules often admit left-induced structures as well.

\begin{theorem}\label{thm:a-mod-cylinder}\cite[Theorem 2.2.3]{hkrs} Let  $(\VV,\otimes, \kk)$ be a locally presentable, closed monoidal model category in which the mon\-oid\-al unit $\kk$ is cofibrant.   If $A$ is a monoid in $\VV$ such that the category $\VV_A$ of right $A$-modules admits underlying-cofibrant replacements (e.g., if all objects of $\VV$ are cofibrant),
then $\VV_A$ admits a model structure left-induced from the forgetful/hom-adjunction
\[ \xymatrix@C=4pc{ \VV_A \ar@<1ex>[r]^U \ar@{}[r]|\perp & \VV. \ar@<1ex>[l]^{\hom(A,-)}}\]
\end{theorem}

\begin{convention}\label{conv:module} Henceforth, we assume always that $\VV$ is a symmetric monoidal model category and that for every algebra $A$, the categories $\VV_{A}$ and ${}_{A}\VV$ of right and left $A$-modules are equipped with model category structures with weak equivalences created in the underlying category $\VV$.
\end{convention}

The tensor product of a right and a left $A$-module over $A$ is construction that appears frequently in this article.

\begin{definition}\label{defn:tensor-over} Given right and left $A$-modules $M_A$ and ${}_AN$, with structure maps $\rho\colon M\tensor A\rightarrow M$ and $\lambda\colon A\tensor N\rightarrow N$, their \emph{tensor product over $A$} is the object $M\tensor_A N$ in $\VV$ defined by the following coequalizer diagram:
$$
\coequalizer{M\tensor A\tensor N}{M\tensor N}{M\tensor_A N}{\rho\tensor 1}{1\tensor \lambda}.
$$
\end{definition}

The special classes of modules defined below, which are characterized in terms of tensoring over $A$, play an important role in this article.

\begin{definition}\label{defn:special-modules} Let $\VV$ be a symmetric monoidal model category satisfying Hypothesis \ref{conv:module}.
A left $A$-module $M$ is called
\begin{itemize}
\item \emph{homotopy flat} if $-\tensor_A M\colon\VV_{A} \to \VV$ preserves weak equivalences;
\item \emph{strongly homotopy flat} if it is homotopy flat and for every finite category $\mathsf J$ and every functor $\Phi\colon \mathsf J \to \VV_{A}$, the natural map
$$(\lim_{\mathsf J} \Phi)\otimes_{A} M \to  \lim_{\mathsf J} (\Phi\otimes_{A} M)$$
is a weak equivalence in $\VV$;
\item \emph{homotopy faithful} if $-\tensor_A M\colon\VV_{A} \to \VV$ reflects weak equivalences;
\item \emph{homotopy faithfully flat} if it is both homotopy faithful and strongly homotopy flat;
\item \emph{homotopy projective} if $\Map_A(M,-)\colon {}_A\VV \to \VV$ preserves weak equivalences;
\item\emph{homotopy cofaithful} if $\Map_A(M,-)\colon {}_A\VV \to \VV$ reflects weak equivalences.
\end{itemize}
Right modules of the same types are defined similarly.
\end{definition}

It is also useful to distinguish those weak equivalences of  left (respectively, right) $A$-modules that remain weak equivalences upon tensoring over $A$ with any right (respectively, left) $A$-module.  

\begin{definition} \label{def:pure we}
A morphism of left $A$-modules $f\colon N\to N'$ is a \emph{pure weak equivalence} if the induced map $M\tensor_{A} f\colon M\tensor_A N\rightarrow M\tensor_A N'$ is a weak equivalence for all cofibrant right $A$-modules $M$.  Pure weak equivalences of right $A$-modules are defined analogously.
\end{definition}

It is easier to work in monoidal model categories in which cofibrant modules are homotopy flat, fitting our intuition of cofibrancy as a sort of projectivity.

\begin{definition}\label{definition:CHF} Let $\VV$ be a symmetric monoidal model category satisfying Hypothesis \ref{conv:module}.
We say that $\VV$ satisfies the \emph{CHF hypothesis} if for every algebra $A$ in $\VV$, every cofibrant right $A$-module is homotopy flat.
\end{definition}

As pointed out in \cite[\S 4]{schwede-shipley}, the CHF hypothesis holds in many monoidal model categories of interest, such as the categories of simplicial sets equipped with usual Kan model structure, symmetric spectra equipped with the stable model structure, (bounded or unbounded) chain complexes over a commutative ring equipped with the projective model structure, and $S$-modules equipped with the usual model structure. The following proposition highlights one of the advantages of this hypothesis.

\begin{proposition} \label{prop:pure we alt}\cite[Proposition 2.16]{berglund-hess:morita}
Let $\VV$ be a symmetric monoidal model category satisfying Hypothesis \ref{conv:module}. If $\VV$ satisfies the CHF hypothesis, then the notions of pure weak equivalence and weak equivalence coincide for modules over any algebra $A$.
\end{proposition}

Our interest in pure weak equivalences is motivated by the next proposition, for which we need to establish a bit of terminology.

\begin{definition} Let $\VV$ be a monoidal category, and let $\varphi\colon A\to B$ be a morphism of algebras in $\VV$.   The restriction/\-extension-of-scalars adjunction,
$$\bigadjunction{\VV_A}{\VV_B}{\vp_{*}}{\vp^*}$$
is defined on objects by $\vp_{*}(M)= M\otimes _{A}B$, endowed with right $B$-action given by multiplication in $B$, for all right $A$-modules $M$, while $\vp^*(N)$ has the same underlying object, but with right $A$-action given by the composite
$$N\otimes A \xrightarrow{N\otimes \vp} N\otimes B \xrightarrow \rho N.$$
\end{definition}

\begin{remark} It is a classical result that $\vp^*$ is right adjoint to $\vp_{*}$.  Moreover, under Hypothesis \ref{conv:module},  the adjunction $\vp_{*}\dashv \vp^*$ is a Quillen pair if and only if $\vp^*$ preserves fibrations.
\end{remark}

We can now formulate a necessary and sufficient condition under which a Quillen pair $\vp_{*}\dashv \vp^*$ is actually a Quillen equivalence. 

\begin{proposition} \label{prop:resext}\cite[Proposition 2.17]{berglund-hess:morita} Let $\VV$ be a symmetric monoidal model category satisfying Hypothesis \ref{conv:module}, and let $\varphi\colon A\to B$ be a morphism of algebras in $\VV$ such that $\vp^*$ preserves fibrations.  The restriction/\-extension-of-scalars adjunction,
$$\bigadjunction{\VV_A}{\VV_B}{\vp_{*}}{\varphi^*},$$
is a Quillen equivalence if and only if $\varphi\colon A\rightarrow B$ is a pure weak equivalence of right $A$-modules.
\end{proposition}

\begin{remark}
The proposition above is a special case of Theorem 2.24 in \cite{berglund-hess:morita}, which provides necessary and sufficient conditions for an adjunction between $\VV_A$ and $\VV_B$ governed by an $A$-$B$-bimodule to be a Quillen equivalence.
\end{remark}

\subsection{Homotopical Morita theory for comodules}

\subsubsection{Review of corings and their comodules}
Let $\VV$ be a monoidal category.  For every  algebra $A$ in $\VV$, the tensor product $-\tensor_A -$ endows the category of $A$-bimodules ${}_A\VV_A$ with a (not necessarily symmetric) monoidal structure, for which the unit is $A$, viewed as an $A$-bimodule over itself.

\begin{definition}
An \emph{$A$-coring} is a coalgebra in the monoidal category $({}_A\VV_A,\tensor_A,A)$, i.e., an $A$-bimodule $C$ together with maps of $A$-bimodules $\Delta\colon C \rightarrow C\tensor_A C$ and $\epsilon\colon C\rightarrow A$, such that the diagrams
$$
\xymatrix{C \ar[r]^-{\Delta} \ar[d]_-{\Delta} & C\tensor_A C \ar[d]^-{C\tensor \Delta} \\ C\tensor_A C \ar[r]^-{\Delta\tensor C} & C\tensor_A C \tensor_A C} \quad \quad \xymatrix{C\ar[r]^-{\Delta} \ar[d]_-{\Delta} \ar@{=}[dr] & C\tensor_A C \ar[d]^-{C\tensor \epsilon} \\ C\tensor_A C\ar[r]_-{\epsilon\tensor C} & C}
$$
are commutative. A \emph{morphism of $A$-corings} is a map of $A$-bimodules $f\colon C\rightarrow D$ such that the diagrams
$$\xymatrix{C\ar[r]^-{\Delta_C} \ar[d]_-f & C\tensor_A C\ar[d]^-{f\tensor_A f} \\ D \ar[r]^-{\Delta_D} & D\tensor_A D}\quad \quad \quad  \xymatrix{C\ar[r]^-{\epsilon_C} \ar[d]_-f & A \ar@{=}[d] \\ D \ar[r]^-{\epsilon_D} & A}$$
commute.
\end{definition}

In Section \ref{sec:rel H-G} we provide natural constructions of families of corings.  For the moment we note only that any algebra $A$ can be seen in a trivial way as a coring over itself, where the comultiplication is the isomorphism $A \xrightarrow{\cong} A\otimes_{A}A$ and the counit is the identity. 

A more general notion of morphism of corings takes into account changes of the underlying algebra as well. Note first that if $\varphi\colon A\rightarrow B$ is a morphism of algebras, then there is a two-sided extension/restriction-of-scalars adjunction,
$$\bigadjunction{{}_A\VV_A}{{}_B\VV_B}{\varphi_*}{\varphi^*},\quad \varphi_* \dashv \varphi^*,$$
where $\varphi_*(M) = B\tensor_A M \tensor_A B$. Moreover, $\varphi_{*}$ is an op-monoidal functor, i.e., there is a natural transformation
$$\varphi_*(M\tensor_A N) \rightarrow \varphi_*(M)\tensor_B \varphi_*(N),$$
which allows us to endow $\varphi_*(C)$ with the structure of a $B$-coring whenever $C$ is an $A$-coring. 

\begin{remark} \label{rmk:descentcoring} Note that if $A$ is considered as an $A$-comodule, where $A$ is equipped with the trivial coring structure defined above, then $\varphi_{*}(A)$ is exactly the well known \emph{descent} or \emph{canonical coring} associated to the algebra morphism $\vp$, with underlying $B$-bimodule $B\otimes_{A}B$. 
\end{remark}

\begin{definition} \label{def:coring}
A \emph{coring} in $\VV$ is a pair $(A,C)$ where $A$ is an algebra in $\VV$, and $C$ is an $A$-coring. A \emph{morphism of corings} $(A,C)\rightarrow (B,D)$ is a pair $(\varphi,f)$ where $\varphi\colon A\rightarrow B$ is a morphism of algebras, and $f\colon \vp_*(C)\rightarrow D$ is a morphism of $B$-corings. The category of corings in $\VV$ is denoted $\Coring_{\VV}$.
\end{definition}

We now recall the definition of a comodule over a coring.

\begin{definition} Let $(A,C)$ be a coring in $\VV$, with comultiplication $\Delta$ and counit $\epsilon$. 
A right \emph{$(A,C)$-comodule} is a right $A$-module $M$ together with a morphism of right $A$-modules
$\delta \colon M\rightarrow M\tensor_A C$ such that the diagrams
$$
\xymatrix{M \ar[r]^-{\delta} \ar[d]_-{\delta} & M\tensor_A C \ar[d]^-{M\tensor \Delta} \\ M\tensor_A C \ar[r]^-{\delta\tensor C} & M\tensor_A C \tensor_A C} \quad \quad \xymatrix{M\ar[r]^-{\delta} \ar@{=}[dr] & M\tensor_A C \ar[d]^-{M\tensor \epsilon} \\ & M}
$$
are commutative. A \emph{morphism of $(A,C)$-comodules} is a morphism $f\colon M\rightarrow N$ of right $A$-modules such that the diagram
$$
\xymatrix{M \ar[r]^-{\delta_M} \ar[d]_-f & M\tensor_A C \ar[d]^-{f\tensor C} \\ N \ar[r]^-{\delta_N} & N\tensor_A C}
$$
commutes.
We let $\VV_A^C$ denote the category of right $(A,C)$-comodules.  The category ${}_{A}^{C}\VV$ of left $(A,C)$-comodules is defined analogously.  
\end{definition}

\begin{remark}\label{rmk:factor} Every morphism of corings $(\vp, f)\colon  (A,C)\to (B,D)$ factors in $\Coring_{\VV}$
as
$$
\xymatrix{(A,C) \ar[rr]^-{(\vp, \Id_{\vp_{*}(C)})} \ar[drr]_-{(\vp,f)} && (B,\vp_{*}(C)) \ar[d]^-{(\Id_{B},f)} \\
&& (B,D),}
$$
i.e., as a change of rings, followed by a change of corings.  This easy observation is a very special case of  \cite[Proposition 3.31]{berglund-hess:morita}.
\end{remark}

There is an adjunction
$$\bigadjunction{\VV_A^C}{\VV_A}{\UU}{-\tensor_A C},\quad \UU_A \dashv -\tensor_A C,$$
where $\UU$ is the forgetful functor, and $-\otimes_{A}C$ is the \emph{cofree $C$-comodule functor}.  In particular, for any $A$-module $M$, the $C$-coaction on $M\otimes_{A}C$ is simply $M\otimes_{A} \Delta$.  

\begin{remark} Note that if $A$ is endowed with its trivial $A$-coring structure, then the adjunction above specializes to an isomorphism between $\VV_A^A$ and $\VV_{A}$.  It follows that the theory of comodules over corings englobes that of modules over algebras.
\end{remark}

Under reasonable conditions on $\VV$,  if $(\varphi,f)\colon (A,C)\rightarrow (B,D)$ is a morphism of corings, then the restriction/extension-of-scalars adjunction on the module categories lifts to an adjunction on the corresponding comodule categories.

\begin{proposition} \label{prop:bbadj}\cite[Proposition 3.16, Example 3.21]{berglund-hess:morita}
Let $\VV$ be a symmetric monoidal category that admits all reflexive coequalizers and coreflexive equalizers.  If $(A,C)$ is a coring in $\VV$ such that $\VV_A^C$ admits all coreflexive equalizers, then
every morphism of corings $(\varphi,f)\colon (A,C)\rightarrow (B,D)$  gives rise to an adjunction
$$\bigadjunction{\VV_A^C}{\VV_B^D}{(\vp,f)_{*}}{(\vp, f)^{*}}$$
such that the following diagram of left adjoints commutes.
$$\xymatrix{
\VV_A^C \ar[rr]^-{(\vp, f)_{*}} \ar[d]_-{\UU} && \VV_B^D \ar[d]^-{\UU} \\
\VV_A \ar[rr]^{\vp_{*}} && \VV_B}
$$
\end{proposition}

\begin{remark}\label{remark:coref-eq} As explained in \cite[Remark 3.9]{berglund-hess:morita}, if $\VV$ is locally presentable, then $\VV_A^C$ admits all coreflexive equalizers. On the other hand,  the dual of \cite[Corollary 3]{linton} implies that if $-\otimes_{A}C\colon   \VV_{A}\to \VV_{A}$ preserves coreflexive equalizers, then $\VV_{A}^{C}$ admits all coreflexive equalizers.
\end{remark}

\begin{remark}  The commutativity of the square in the statement of Proposition \ref{prop:bbadj} implies that for any $C$-comodule $(M,\delta)$, the $B$-module underlying $(\vp, f)_{*}(M, \delta)$ is $M\otimes _{A}B$.  As shown in the proof of \cite[Proposition 3.16]{berglund-hess:morita} (in a somewhat more general context), the $D$-coaction on $M\otimes_{A}B$ is given by the following composite.
{\footnotesize$$\xymatrix{M\otimes _{A}B \ar [r]^(0.25){\delta\otimes B}&M\otimes_{A}C\otimes_{A}B\cong M\otimes _{A}A\otimes_{A}C\otimes_{A}B\ar [rr]^(0.6){M\otimes \vp \otimes C\otimes B}&&M\otimes_{A}B\otimes_{A}C\otimes_{A}B \ar[d]^{M\otimes f}\\
&&& M\otimes _{A}D\ar [d]^{\cong}\\
&&& M\otimes_{A}B\otimes_{B}D}$$}

Since the diagram of right adjoints must also commute, we know as well that the image under $(\vp, f)^{*}$ of a cofree $D$-comodule $N\otimes _{B}D$ is the cofree $C$-comodule $\vp^{*}(N)\otimes _{A}C$. 
\end{remark}

\begin{notation} When $(\vp,f)=(\Id_{A},f)\colon(A,C)\to (A,D)$,   we denote the induced adjunction
\begin{equation}\label{eq:corestriction}
\bigadjunction{\VV_{A}^{C}}{\VV_{A}^{D}}{f_{*}}{f^{*}}
\end{equation}
and call it the \emph{coextension/corestriction-of-coefficients adjunction} or \emph{change-of-corings adjunction} associated to $f$.  Note that the $D$-component of the counit of the $f_{*}\dashv f^{*}$ adjunction  is $f$ itself and that for every $(A,C)$-comodule $(M, \delta)$,
$$f_{*}(M,\delta)= \big(M, (1\otimes f)\delta\big).$$

When $(\vp, f)=(\vp, \Id_{\vp_{*}(C)})\colon (A,C)\to \big(B,\vp_{*}(C)\big)$, we denote the induced adjunction   
\begin{equation}\label{eq:canonical}
\bigadjunction{\VV_{A}^{C}}{\VV_{B}^{\vp_{*}(C)}}{\Can_{\vp}}{\Prim_{\vp}}
\end{equation}
and call it the \emph{canonical adjunction for $C$}, as a generalization of the usual canonical adjunction for descent  along $\vp\colon  A\to B$, which is the case $C=A$ of the adjunction above \cite{hess:descent}, \cite{mesablishvili}.
\end{notation}

\begin{remark}\label{rmk:factor2} By Remark \ref{rmk:factor}, the adjunction $(\vp, f)_{*}\dashv (\vp, f)^{*}$ can be factored as follows.
\begin{equation} \label{eq:braided factorization}
\xymatrix{ \VV_A^C \ar@<1ex>[rrr]^-{\Can _{\vp}} &&& \VV_B^{\vp_*(C)} \ar@<1ex>[lll]^-{\Prim_{\vp}} \ar@<1ex>[rrr]^-{f_*} &&& \VV_B^D. \ar@<1ex>[lll]^-{f^{*}}}
\end{equation}
\end{remark}

The right adjoint $(\vp, f)^{*}$ in the adjunction governed by a morphism of corings $(\vp, f)\colon  (A,C)\to (B,D)$ is difficult to describe in general. Under appropriate conditions on the left $A$-module underlying $C$, however, it is possible to express $(\vp, f)^*$ as a cotensor product over $D$, dually to the expression of the left adjoint in the extension/restriction-of-scalars adjunction associated to $\vp$ as a tensor product over $A$. The condition we need to impose on $C$ is formulated as follows.

\begin{definition}\label{definition:flat-coring}
A coring $(A,C)$ is \emph{flat}  if $-\tensor_A C\colon \VV_A\to \VV_A$ preserves coreflexive equalizers.
\end{definition}

Flatness  of a coring gives us control of coreflexive equalizers in the associated comodule category.

\begin{proposition}\cite[Proposition 3.28]{berglund-hess:morita}
If $(A,C)$ is a flat coring, then the forgetful functor $\UU\colon \VV_A^C\to \VV_A$ creates coreflexive equalizers. 
\end{proposition}

The following definition is dual to Definition \ref{defn:tensor-over}.

\begin{definition}\label{defn:cotensor-over}
Suppose that the monoidal category $\VV$ admits coreflexive equalizers. Let $(A,C)$ be a coring in $\VV$, let $M$ be a right and $N$ a left $(A,C)$-comodule. The \emph{cotensor product} $M\cotensor_C N$ is defined as the coreflexive equalizer in $\VV$: 
$$
\equalizer{M\cotensor_C N}{M\tensor_A N}{M\tensor_A C \tensor_A N.}{\;\delta_M\tensor N}{\negthinspace\negthinspace M\tensor \delta_N}
$$
\end{definition}

We can now formulate the desired explicit description of the right adjoint in the adjunction governed by a morphism of corings.

\begin{proposition} \label{prop:R_X dualizable}\cite[Proposition 3.30]{berglund-hess:morita}
Let $\VV$ be a  monoidal category admitting all reflexive coequalizers and coreflexive equalizers. Let $(A,C)$ be a flat coring in $\VV$.
If $(\vp,f)\colon (A,C)\rightarrow (B,D)$ is a coring morphism, then $B\tensor_A C$ admits the structure of a left $(B,D)$-comodule in $\VV_A^C$ such that the functor $(\vp,f)^{*}$ is isomorphic to the cotensor product functor $-\cotensor_{D}(B \tensor_A C)$, i.e., there is an adjunction
$$
\hugeadjunction{\VV_A^C}{\VV_B^D}{(\vp,f)_{*}}{-\cotensor_{D} (B\tensor_A C)}.
$$
\end{proposition}

\begin{remark}  The left $D$-coaction on $B\otimes_{A}C$ is given by the following composite.
{\small$$\xymatrix{B\otimes _{A}C \ar [r]^(0.25){B\otimes \Delta}&B\otimes_{A}C\otimes _{A}C\cong B\otimes _{A}C\otimes_{A}A\otimes_{A}C\ar [rr]^(0.6){B\otimes C\otimes \vp \otimes C}&&B\otimes _{A}C\otimes_{A}B\otimes_{A}C \ar[d]^{f\otimes C}\\
&&& D\otimes _{A}C\ar [d]^{\cong}\\
&&& D\otimes_{B}B\otimes_{A}C}$$}
\end{remark}

\subsubsection{Homotopy theory of comodules}
We now introduce homotopy theory into our discussion of comodule categories.  

\begin{convention} \label{conv:m-t} We assume henceforth that $\VV$ is a symmetric monoidal model category satisfying Hypothesis \ref{conv:module}.
For every coring $(A,C)$  in $\VV$ that we consider here, we suppose moreover   that $\VV_A^C$  admits the model category structure left-induced from $\VV_{A}$, via the adjunction
$$\bigadjunction{\VV_A^C}{\VV_A}{\UU_A}{-\tensor_A C}.$$
\end{convention}

\begin{remark}
Conditions on $\VV$ under which the convention above holds can be found in \cite{bhkkrs}, \cite{hkrs}, and \cite{hess-shipley}, where a number of concrete examples, including simplicial sets, symmetric spectra, and chain complexes, are also treated. In Section \ref{sec:ch cx} we recall in detail the example of unbounded chain complexes over a commutative ring.
\end{remark}

\begin{remark}\label{rmk:QA-coring map} It follows from \cite[Proposition 4.5]{berglund-hess:morita} that if  $\VV$ is a symmetric monoidal model category satisfying Hypothesis \ref{conv:m-t}, and  $(\vp, f):(A,C) \to (B,D)$ is a morphism of corings, then the associated adjunction
$$\bigadjunction{\VV_A^C}{\VV_B^D}{(\vp, f)_{*}}{(\vp, f)^{*}}$$
is a Quillen adjunction if
$$\bigadjunction{\VV_A}{\VV_B}{\vp_{*}}{\vp^{*}}$$
is, i.e., if $\vp^{*}$ preserves fibrations.  In particular, for every morphism of corings $(1,f): (A, C)\to (A,D)$, 
$$\bigadjunction{\VV_A^C}{\VV_A^D}{ f_{*}}{f^{*}}$$
is a Quillen adjunction.
\end{remark}

\begin{remark}
Since we assume henceforth that $\VV$, $\VV_{A}$, and $\VV_A^C$ are model categories, they are in particular complete and cocomplete and thus admit all reflexive coequalizers and coreflexive equalizers.
\end{remark}

We now recall from \cite[Section 4]{berglund-hess:morita} the conditions under which a morphism of corings $(\vp, f)\colon   (A,C)\to (B,D)$ induces a Quillen equivalence of the associated comodule categories. We begin by breaking the problem into two pieces, according to the factorization in Remark \ref{rmk:factor2}.

\begin{definition} \label{defn:homotopic descent}
Let $(A,C)$ be a coring in $\VV$ and $B$ an algebra in $\VV$. An algebra morphism $\vp\colon A \to B$ \emph{satisfies effective homotopic descent with respect to $C$} if the adjunction
\begin{equation}\label{eqn:canonical-adj}
\bigadjunction{\VV_A^C}{\VV_B^{\vp_*(C)},}{\Can_{\vp}}{\Prim_{\vp}}
\end{equation}
is a Quillen equivalence.
\end{definition}

Sufficient conditions for effective homotopic descent were established in \cite{berglund-hess:morita}.

\begin{proposition}\label{prop:descent}\cite[Corollary 4.12]{berglund-hess:morita} Let $\VV$ be a symmetric monoidal model category satisfying Hypothesis \ref{conv:m-t}. 
Let $\varphi\colon A\rightarrow B$ be a morphism of algebras in $\VV$. If $B$ is homotopy faithfully flat as a left $A$-module, then $\varphi$ satisfies effective homotopic descent.
\end{proposition}

For the other piece of the factorization, we need to introduce a notion dual to that of pure weak equivalence.

\begin{definition} \label{defn:copure}
Let $\VV$ be a symmetric monoidal model category satisfying Hypothesis \ref{conv:m-t}. We say that a map $f\colon C\to D$ of $A$-corings is a \emph{copure weak equivalence} if
$$\epsilon_{M}\colon  f_{*}f^{*}(M) \to M$$ 
is a weak equivalence for all fibrant right $D$-comodules $M$.
\end{definition}

Just as pure weak equivalences induce Quillen equivalences of module categories, copure weak equivalences do the same for comodule categories.

\begin{proposition}\label{prop:we-QE}\cite[Proposition 4.5]{berglund-hess:morita}  Let $\VV$ be a symmetric monoidal model category satisfying Convention \ref{conv:m-t}. Let $A$ be an algebra in $\VV$.
The change-of-corings adjunction,
$$
\bigadjunction{\VV_A^C}{\VV_A^D,}{f_*}{f^{*}}
$$
is a Quillen equivalence if and only if $f\colon C\to D$ is a copure weak equivalence of $A$-corings.
\end{proposition}

\begin{remark}
As pointed out in \cite[Proposition 4.6]{berglund-hess:morita}, if $A$ is fibrant as an object of $\VV$, then every copure weak equivalence of $A$-corings is a weak equivalence. Conversely, if the coring $(A,C)$ is flat, then $f^{*}(M)=M\cotensor_{D}C$ by Proposition \ref{prop:R_X dualizable}. In this case, if $f$ is a weak equivalence, and the functor $M\cotensor_{D}-\colon {}_{A}^{D}\VV \to \VV$ preserves weak equivalences for all fibrant right $D$-comodules $M$, then the adjunction above is a Quillen equivalence. It follows that if every fibrant $D$-module is ``homotopy coflat'', then every  weak equivalence of corings with flat domain is copure; compare with Proposition \ref{prop:pure we alt}. 
\end{remark}

As a consequence of Propositions \ref{prop:descent} and \ref{prop:we-QE}, we obtain the following sufficient condition for the adjunction induced by a coring morphism to be a Quillen equivalence.

\begin{corollary}\label{cor:QE conditions} Let $\VV$ be a symmetric monoidal model category satisfying Convention \ref{conv:m-t}.  Let $(\vp, f)\colon (A,C) \to (B,D)$ be a morphism of corings in $\VV$.

If $B$ is homotopy faithfully flat as a left $A$-module, and $f$ is a copure weak equivalence, then 
$$\bigadjunction{\VV_A^C}{\VV_B^D}{(\vp,f)_{*}}{(\vp, f)^{*}}$$
is a Quillen equivalence.
\end{corollary}

\section{Relative Hopf-Galois extensions}\label{sec:rel H-G}

Here we apply the results of \cite{berglund-hess:morita}   recalled in the previous section to elaborating interesting and natural generalizations first of the classical framework, then of the homotopic  framework, for Hopf-Galois extensions.    

\subsection{The descent and Hopf functors}
Let $(\VV, \otimes, R )$ be a symmetric monoidal category that is both complete and cocomplete. 
Generalizing somewhat constructions in \cite{brzezinski-wisbauer}, we begin by describing two important, natural ways to create corings in $\VV$  and  the relation between these constructions. 

\begin{definition}  Let $\Alg_{\VV}^{\to}$ denote the category of morphisms of algebras in $\VV$.  The \emph{descent functor} $$\Desc \colon  \Alg_{\VV}^{\to} \to \Coring_{\VV}$$ 
sends an object $\vp\colon  A\to B$ to its associated \emph{canonical descent coring} (also called the \emph{Sweedler coring})
$$\Desc(\vp)=\big(B, (B \otimes _{A}B, \Delta _{\vp}, \ve_{\vp})\big),$$
where $\Delta_{\vp}$ is equal to the composite
$$B\otimes _{A}B \cong B\otimes _{A}A \otimes_{A}B\xrightarrow {B\otimes_{A}\vp\otimes_{A}B}B\otimes _{A}B \otimes_{A}B\cong (B\otimes _{A}B)\otimes _{B}(B \otimes_{A}B),$$
and 
$$\ve_{\vp}=\bar\mu _{B}\colon   B\otimes_{A}B \to B,$$ 
the morphism induced by the multiplication $\mu_{B}\colon  B\otimes B \to B$.  A morphism $(\alpha,\beta)\colon  \vp \to \vp'$ in $\Alg^{\to}$, i.e.,
a commuting diagram of algebra morphisms
$$\xymatrix{A\ar [d]_{\vp}\ar [r]^{\alpha}&A'\ar [d]^{\vp'}\\ 
B\ar [r]^{\beta}&B',}$$
induces a morphism of $B$-corings
$$\Desc (\alpha,\beta)=(\beta, \beta\otimes _{\alpha}\beta)\colon   \Desc (\vp)\to \Desc(\vp').$$
\end{definition}

\begin{remark}  The coring $\Desc(\vp)$ is the same as the coring $\vp_{*}(A)$ of Remark \ref{rmk:descentcoring}.  We change the notation here to emphasize the functoriality of the construction in the morphism $\vp$.
\end{remark}

It is not hard to check that $\big(B, (B \otimes _{A}B, \Delta _{\vp}, \ve_{\vp})\big)$ is indeed a $B$-coring and that $(\beta, \beta\otimes _{\alpha}\beta)$ is a morphism of corings for any $(\alpha,\beta)\colon  \vp \to \vp'$. Moreover, $\Desc(\vp)$ admits two natural coaugmentations, given by the composites
$$B\cong B\otimes _{A}A \xrightarrow {B\otimes_{A}\vp} B\otimes _{A}B\quad \text{and}\quad B\cong A\otimes _{A}B \xrightarrow {\vp\otimes_{A}B} B\otimes _{A}B.$$

The other functor into $\Coring_{\VV}$ that we consider here takes as input algebras, respectively coalgebras, endowed with extra structure given by a bialgebra $H$.

\begin{remark} If $H$ is a bialgebra in $\VV$, then it is an algebra in $\Coalg_{\VV}$ and a coalgebra in $\Alg_{\VV}$.  In particular, $( R , H)$ is a coring in ${\VV}$.
\end{remark}

\begin{definition} Let $H$ be a bialgebra in $\VV$. An object of the category $(\Coalg_{\VV})_{H}$ of  \emph{$H$-module coalgebras} in $\VV$ is an $H$-module in $\Coalg_{\VV}$, i.e., a coalgebra $C$ in $\VV$, equipped with an associative, unital morphism of coalgebras $\kappa\colon   C\otimes H \to C$.  Morphisms in $(\Coalg_{\VV})_{H}$ are morphisms in $\VV$ that respect the comultiplication and counit and the $H$-action.
\end{definition}

\begin{definition}  Let $H$ be a bialgebra in $\VV$. An object of the category $\Alg_{\VV}^{H}$ of  \emph{$H$-comodule algebras} in $\VV$ is an $H$-comodule in $\Alg_{\VV}$, i.e., an algebra $A$ in $\VV$, equipped with a coassociative, counital morphism of algebras $\rho\colon   A \to A\otimes H$.  Morphisms in $\Alg_{\VV}^{H}$ are morphisms in $\VV$ that respect the multiplication and unit and the $H$-coaction.
\end{definition}

\begin{notation} Let $\Gamma\colon   H\to H'$ be a morphism of bialgebras. There is an induced extension/restriction-of-scalars adjunction
$$\bigadjunction{(\Coalg_{\VV})_{H}}{(\Coalg_{\VV})_{H'}.}{\Gamma_{*}}{\Gamma^{*}}$$
Moreover, by Proposition \ref{prop:bbadj}  there is also a change-of-corings adjunction
$$\bigadjunction{\Alg_{\VV}^{H}}{\Alg_{\VV}^{H'}.}{\Gamma_{*}}{\Gamma^{*}}$$
As we are using the same notation for the functors in these two different cases, we will be very careful to specify context any time we refer to a functor $\Gamma_{*}$ or $\Gamma^{*}$.
\end{notation}

The category below of matched pairs of comodule algebras and module coalgebras is the natural domain for an interesting functor to the global category $\Coring _{\VV}$ of all corings in $\VV$, generalizing the well known construction of a coring from any comodule algebra \cite [Example 4.3(2)]{hess:hhg}.

\begin{definition} The category $\Pair_{\VV}$ has as objects triples $\big(H, (A, \rho), (C, \kappa)\big)$, where $H$ is a bialgebra in $\VV$, $(A, \rho)$ is an $H$-comodule algebra, and $(C,\kappa)$ is an $H$-module coalgebra. A morphism  from  $\big(H, (A, \rho_{A}), (C,\kappa_{C})\big)$ to $\big(K, (B, \rho_{B}), (D, \kappa_{D})\big)$ consists of a triple $(\Gamma, \vp, \theta)$, where $\Gamma\colon   H\to K$ is a morphism of bialgebras, $\vp\colon \Gamma_{*}(A) \to B$ is a morphism of $K$-comodule algebras, and $\theta\colon   C\to \Gamma^{*}(D)$ is a morphism of $H$-module coalgebras.
\end{definition}

\begin{definition}\label{defn:hopf} Let $H$ be a bialgebra in $\VV$. The \emph{Hopf functor} 
$$\Hopf\colon   \Pair_{\VV}\to \Coring_{\VV}$$ 
sends an object $\big(H,(A, \rho), (C,\kappa)\big)$ to its associated \emph{Hopf coring},
$$\Hopf(\rho, \kappa)=\big( A, (A\otimes C, \Delta_{\rho,\kappa}, \ve_{\rho, \kappa})\big),$$
where the left $A$-action is equal to
$$A\otimes A\otimes C \xrightarrow {\mu\otimes C} A\otimes C,$$
where $\mu$ is the multiplication on $A$, and the right $A$-action is given by the composite 
$$A\otimes C \otimes A \xrightarrow{A\otimes C\otimes \rho} A\otimes C \otimes A\otimes H\cong A\otimes A \otimes C\otimes H\xrightarrow {\mu\otimes \kappa} A\otimes C.$$
The comultiplication $\Delta_{\rho,\kappa}$ is equal to the composite
$$A\otimes C\xrightarrow {A\otimes \Delta } A\otimes C\otimes C \cong (A\otimes C)\otimes _{A} (A\otimes C),$$
where $\Delta$ is the comultiplication on $C$, and $\ve_{\rho, \kappa}$ is given by
$$ A\otimes C \xrightarrow {A\otimes \ve} A\otimes R  \cong A,$$
where $\ve$ is the counit of $C$.  

If $(\Gamma, \vp, \theta)$ is a morphism from $\big(H, (A, \rho_{A}), (C,\kappa_{C})\big)$ to $\big(K, (B, \rho_{B}), (D, \kappa_{D})\big)$, then the morphism of $A$-bimodules underlying $\Hopf(\Gamma, \alpha, \theta)$ is
$$\vp\otimes \theta\colon   A\otimes C\to B\otimes D.$$ 
\end{definition}

The proof that $\Hopf(\rho, \kappa)$ is actually an $A$-coring is somewhat fastidious, but straightforward.

\begin{notation}  An important special case of the construction above comes from taking $(C,\kappa)=(H, \mu)$, where $\mu$ is the multiplication on $H$.  We simplify notation a bit and write
$$\Hopf(\rho)=\Hopf(\rho, \mu).$$
\end{notation}

The relation between the functors $\Desc$ and $\Hopf$ can be expressed in terms of a natural transformation, as explained below. Observe first that the proof of \cite[Proposition 4.3]{milnor-moore} can easily be generalized to an arbitrary monoidal category, implying that for any morphism $\Gamma\colon   H\to K$ of bialgebras, the coequalizer $R \otimes _{H} K$ inherits a coalgebra structure from $K$, with compatible right $K$-module structure, induced by the multiplication in $K$.

\begin{notation}  If $\Gamma\colon H\to K$ is a morphism of bialgebras in $\VV$, let $\Cof(\Gamma)$ denote the $K$-module coalgebra $R \otimes_{H}K$, let $\bar \mu_{K}$ denote its induced right $K$-action, and let $\pi_{\Gamma}\colon K\to \Cof(\Gamma)$ denote the quotient map.
\end{notation}

\begin{definition} 
The category $\Comodalg_{\VV}$ of all comodule algebras in $\VV$ has as objects pairs $\big(H, (A, \rho)\big)$, where $H$ is a bialgebra in $\VV$, and $(A, \rho)$ is an $H$-comodule algebra. A morphism in  $\Comodalg_{\VV}$ from  $\big(H, (A, \rho_{A})\big)$ to $\big(K, (B, \rho_{B})\big)$ consists of a pair $(\Gamma, \vp)$, where $\Gamma\colon H\to K$ is a morphism of bialgebras, and $\vp\colon \Gamma_{*}(A) \to B$ is a morphism of $K$-comodule algebras.
\end{definition}

\begin{definition} Let $\Comodalg_{\VV}^{\to}$ denote the category of morphisms in the category $\Comodalg_{\VV}$. Let 
$$U^{\to}\colon \Comodalg_{\VV}^{\to} \to \Alg_{\VV}^{\to}\colon \Big( \big(H,(A, \rho_{A})\big) \xrightarrow{(\Gamma, \vp)} \big(K,(B, \rho_{B})\big)\Big) \mapsto \big( A \xrightarrow{\vp} B)$$
be the obvious forgetful functor, and
\begin{align*}
C\colon  \Comodalg&_{\VV}^{\to}\to \Pair_{\VV}\colon   \\
&\Big( \big(H,(A, \rho_{A})\big) \xrightarrow{(\Gamma, \vp)} \big(K,(B, \rho_{B})\big)\Big) \mapsto \big(K,(B, \rho_{B}), (\Cof(\Gamma), \bar\mu_{K})\big)
\end{align*}
the ``cofiber'' functor.  

The \emph{Galois transformation} is the natural transformation 
$$\Gal\colon   \Desc \circ U^{\to} \to \Hopf \circ C$$ 
defined on an object $\big(H,(A, \rho_{A})\big) \xrightarrow{(\Gamma, \vp)} \big(K,(B, \rho_{B})\big)$ so that 
$$\Gal_{(\Gamma, \vp)}\colon   \Desc(\vp) \to \Hopf(\rho_{B}, \overline \mu_{K})$$ 
is the morphism of $B$-corings given by the identity on $B$ in the algebra component and by the composite
$$B\otimes_{A}B \xrightarrow{B\otimes _{A}\rho_{B}} B\otimes _{A} B\otimes K\xrightarrow{\bar\mu_{B}\otimes \pi_{\Gamma}} B\otimes \Cof(\Gamma)$$
in the coring component, where $\bar \mu_{B}$ is induced by the multiplication in $B$; compare with the Galois map of Definition \ref{defn:hg-classical}.
\end{definition}

The diagram below summarizes the definitions seen thus far in this section.

$$\xymatrix{\Alg_{\VV}^{\to}\ar@<0.5ex> [ddrr]^{\Desc}\\
\\
\Comodalg_{\VV}^{\to}\ar[uu]^{U^{\to}}\ar [dd]_{C}\ar@<1.5ex>[rr]\ar@<-1.5ex>[rr]&\Downarrow\scriptstyle{\Gal}&\Coring_{\VV}\\
\\
 \Pair_{\VV}\ar@<-0.5ex>[uurr]_{\Hopf}}$$
 
 \begin{remark}  An object in $\Comodalg_{\VV}^{\to}$ of the form 
 $$\big(R ,(A, \rho_{A})\big) \xrightarrow{(\eta, \vp)} \big(K,(B, \rho_{B})\big),$$
where $\eta\colon  R \to K$ is the unit  of $K,$  is \emph{Hopf-Galois data}, in the sense of \cite{hess:hhg}, since one can also view a morphism of this type as a morphism of $K$-comodule algebras, where the coaction of $K$ on $A$ is trivial.
 \end{remark}
 
\begin{remark}\label{rmk:Galois functoriality}  The naturality of all of the constructions seen thus far implies that a commuting diagram of comodule algebra morphisms
$$\xymatrix{(H,A) \ar[d]_{(\zeta, \alpha)}\ar[r]^{(\Gamma, \vp)}&(K,B)\ar[d]_{(\xi, \beta)}\\
(H',A') \ar[r]^{(\Gamma', \vp')}&(K',B')}$$
gives rise to a commuting diagram of functors
$$\xymatrix{\VV _{A}^{}\ar @<0.5ex>[rr]^{\alpha_{*}}\ar @<-0.5ex>[dd]_{\Can_{\vp}}&&\VV _{A'}^{}\ar @<0.5ex>[ll]^{\alpha^{*}}\ar @<-0.5ex>[dd]_{\Can_{\vp'}}\\
\\
\VV _{B}^{\Desc(\vp)}\ar @<-0.5ex>[uu]_{\Prim_{\vp}}\ar@<0.5ex>[rr]^{\big(\beta,\Desc (\alpha, \beta)\big)_{*}}\ar @<-0.5ex>[dd]_{\Gal(\Gamma, \vp)_{*}}&&\VV _{B'}^{\Desc(\vp')}\ar @<0.5ex>[ll]^{\big(\beta,\Desc (\alpha, \beta)\big)^{*}}\ar @<-0.5ex>[dd]_{\Gal(\Gamma', \vp')_{*}}\ar @<-0.5ex>[uu]_{\Prim_{\vp'}}\\
\\
\VV _{B}^{\Hopf(\rho_{B},\bar\mu_{K})}\ar @<-0.5ex>[uu]_{\Gal(\Gamma, \vp)^{*}}\ar @<0.5ex>[rr]^{(\beta, \theta_{\xi, \beta})_{*}}&&\VV _{B'}^{\Hopf(\rho_{B'},\bar\mu_{K'})},\ar @<0.5ex>[ll]^{(\beta, \theta_{\xi, \beta})^{*}}\ar @<-0.5ex>[uu]_{\Gal(\Gamma', \vp')^{*}}
}$$
where the $B'$-bimodule map underlying $\theta_{\xi, \beta}$ is
$$B'\otimes \Cof (\Gamma)\otimes _{B}B'\xrightarrow{B'\otimes \Cof (\xi)\otimes B'} B'\otimes \Cof (\Gamma')\otimes _{B}B'\to B'\otimes \Cof (\Gamma'),$$
with the second map given by the right $B'$-action on $B'\otimes \Cof(\Gamma')$ (cf.~Definition \ref{defn:hopf}).
\end{remark}

We need to introduce one more functor defined on $\Comodalg_{\VV}$, in order to set the stage for Hopf-Galois extensions and their generalizations.  

\begin{remark}  Proposition \ref{prop:bbadj} implies that if $\Alg_{\VV}$ admits all reflexive coequalizers and coreflexive equalizers and $\Alg_{\VV}^{H}$ admits all coreflexive equalizers, then every morphism of bialgebras $\Gamma\colon   H\to K$ gives rise to an adjunction
$$\bigadjunction{\Alg_{\VV}^H}{\Alg_{\VV}^K.}{\Gamma_{*}}{\Gamma^{*}}$$
See Remark \ref{remark:coref-eq} for conditions under which these hypotheses hold.  In particular, if $\VV$ is locally presentable, then both $\Alg_{\VV}$ and $\Alg_{\VV}^{H}$ are locally presentable and therefore complete and cocomplete. 
\end{remark}

\begin{definition} Let $H$ be a bialgebra in $\VV$ with unit $\eta\colon   R\to H$.  If the extension-of-corings  functor $\eta_{*}\colon  \Alg_{\VV}\to \Alg_{\VV}^{H}$, which endows any algebra with a trivial $H$-coaction, admits a right adjoint,  then we call this right adjoint the \emph{$H$-coinvariants functor} and denote it
$$(-)^{{\mathrm{co} H}}\colon   \Alg_{\VV}^{H}\to \Alg_{\VV}.$$
\end{definition}

\begin{remark}\label{remark:coinvariants}  Suppose that $\Alg_{\VV}$ admits all reflexive coequalizers and coreflexive equalizers and $\Alg_{\VV}^{H}$ admits all coreflexive equalizers. For any morphism $\Gamma\colon   H\to K$ of bialgebras, there is a commuting diagram of adjunctions, with right adjoints on the inner triangle and left adjoints on the outer triangle, 
\begin{equation}\label{eqn:diag-triangle}\xymatrix{\Alg_{\VV}^{H}\ar@<0.7ex>[rrrr]^{\Gamma_{*}}\ar@<0.5ex> [drr]^{(-)^{\mathrm{co} H}}&&&&\Alg_{\VV}^{K}\ar@<0.5ex> [llll]^{\Gamma^{*}}\ar@<-0.5ex> [dll]_{(-)^{\mathrm{co} K}}\\
&&\Alg_{\VV}\ar @<0.7ex>[ull]^{(\eta_{H})_{*}}\ar @<-0.7ex>[urr]_{(\eta_{K})_{*}}}
\end{equation}
since $\Gamma \circ \eta_{H}=\eta_{K}$.  

Let  $(\Gamma, \vp)\colon  \big(H,(A, \rho_{A})\big) \xrightarrow{} \big(K,(B, \rho_{B})\big)$ be a morphism in
in $\Comodalg_{\VV}$. Recall that if $\eta^{\Gamma}$ is the unit of the $\Gamma_{*}\dashv \Gamma^{*}$-adjunction in diagram (\ref{eqn:diag-triangle}), then the transpose of $\vp\colon   \Gamma_{*}A\to B$ is the composite
$$A \xrightarrow {\eta^{\Gamma}_{A}} \Gamma^{*}\Gamma_{*}A \xrightarrow {\Gamma^{*}\vp} \Gamma^{*}B.$$
Applying $(-)^{{\mathrm{co} H}}$, we obtain a morphism of algebras 
$$A^{\mathrm{co} H}\xrightarrow {(\eta^{\Gamma}_{A})^{\mathrm{co} H}} (\Gamma^{*}\Gamma_{*}A)^{\mathrm{co} H}\xrightarrow {(\Gamma^{*}\vp)^{\mathrm{co} H}} (\Gamma^{*}B)^{\mathrm{co} H}\cong B^{\mathrm{co} K},$$
where the last isomorphism follows from the commutativity of the diagram above.  
We denote this composite morphism 
$$\vp^{\mathrm{co} \Gamma}\colon   A^{\mathrm{co} H}\to B^{\mathrm{co} K},$$
which becomes simply $\vp^{\mathrm{co} H}$ when $\Gamma$ is the identity morphism on $H$. 
\end{remark}

As the constructions above are clearly natural in both the bialgebra and the algebra components of a comodule algebra, we can summarize the discussion above as follows.

\begin{proposition}  Let $\VV$ be a symmetric monoidal category such that $\Alg_{\VV}$ admits all reflexive coequalizers and coreflexive equalizers and $\Alg_{\VV}^{H}$ admits all coreflexive equalizers for all bialgebras $H$. There is a functor $\Coinv\colon    \Comodalg_{\VV}\to \Alg_{\VV}$  that to a morphism $(\Gamma, \vp)\colon  \big(H,(A, \rho_{A})\big) \xrightarrow{} \big(K,(B, \rho_{B})\big)$ in $\Comodalg_{\VV}$ associates the algebra morphism $\vp^{\mathrm{co} \Gamma}\colon   A^{\mathrm{co} H}\to B^{\mathrm{co} K}$.
\end{proposition}

\subsection{The classical Hopf-Galois framework}

We have now set up the complete framework enabling us to formulate a relative version of the classical notion of Hopf-Galois extensions of rings and algebras.  To simplify notation, we drop henceforth the coactions from the notation for comodule algebras.

\begin{definition} \label{def:relhgext} Let $\VV$ be a symmetric monoidal category such that $\Alg_{\VV}$ admits all reflexive coequalizers and coreflexive equalizers and $\Alg_{\VV}^{H}$ admits all coreflexive equalizers for all bialgebras $H$. 
A morphism  $(H,A) \xrightarrow{(\Gamma, \vp)}(K,B)$ in $\Comodalg_{\VV}$ is a \emph{relative Hopf-Galois extension} if 
$$\vp^{\mathrm{co} \Gamma}\colon   A^{\mathrm{co} H}\to B^{\mathrm{co} K}$$
is an isomorphism of algebras,
and
$$\Gal_{(\Gamma, \vp)}\colon   \Desc(\vp) \to \Hopf(\rho_{B},  \bar\mu_{K})$$
is an isomorphism of $B$-corings.
\end{definition}

When $\VV$ is the category of $R $-modules for some commutative ring $R $, a relative Hopf-Galois extension for $H=R $ is exactly a classical Hopf-Galois extension, as defined by Chase and Sweedler \cite{chase-sweedler}.  Related notions of relative Hopf-Galois extensions have been considered in \cite{schauenburg-schneider} and \cite{schneider}, in the context of quotient theory of noncommutative Hopf algebras.

\begin{example}\label{example:normal} Let $H$ be a bialgebra in $\VV$ with unit $\eta$, comultiplication $\Delta$, and multiplication $\mu$.  The morphism $(\eta, \eta)\colon (R , R )\to (H,H)$ in $\Comodalg_{\VV}$ is a relative Hopf-Galois extension if and only if 
$$H\otimes H \xrightarrow {H\otimes \Delta} H\otimes H\otimes H\xrightarrow {\mu\otimes H} H\otimes H$$
is an isomorphism.  If $\VV$ is the category of $R $-modules for some commutative ring $R $, then this condition is equivalent to requiring that $H$ admit an antipode, i.e., that $H$ be a Hopf algebra, in the classical sense of the word \cite[Example 2.1.2]{schauenburg}.  If $\VV$ is the category of (differential) \emph{graded} $R $-modules, then, as is well known, every connected bialgebra $H$ satisfies the condition above \cite[Proposition 3.8.8]{hazewinkel-etal}.  

Inspired by the classical case, we make the following definition.
\begin{definition}
We say that a bialgebra $H$ in $\VV$ is a \emph{Hopf algebra} if the map
$$(\eta, \eta)\colon  (R , R )\to (H,H)$$
is a relative Hopf-Galois extension in the sense of Definition \ref{def:relhgext}. More generally, we say that a morphism of bialgebras $\Gamma\colon H\to K$ is a \emph{relative Hopf algebra} if
$$(\Gamma, \Gamma)\colon (H,H) \to (K,K)$$
is a relative Hopf-Galois extension, i.e., if
\begin{equation}\label{eqn:relHopf}K\otimes _{H}K \xrightarrow {K\otimes _{H}\Delta_{K}} K\otimes _{H}K\otimes K\xrightarrow {\bar \mu_{K}\otimes \pi_{\Gamma}} K\otimes \Cof(\Gamma)
\end{equation}
is a isomorphism.
\end{definition}
For example, if $H$ is any bialgebra, and $H'$ is a Hopf algebra, then the bialgebra morphism $H\otimes \eta'\colon   H\to H\otimes H'$ is a relative Hopf algebra. 

 If $\VV$ is the category of (differential) graded $R$-modules for some commutative ring $R$, then a morphism $\Gamma\colon H\to K$ of bialgebras is a relative Hopf algebra if the left $H$-module and right $\Cof(\Gamma)$-comodule underlying $K$ is isomorphic to $H\otimes \Cof(\Gamma)$.  By \cite[Theorem 4.4]{milnor-moore}, $K$ admits such a description if $H$  and $K$ are connected, while $\Gamma\colon  H\to K$ is split injective and $\pi_{\Gamma}\colon  K \to \Cof(\Gamma)$ is split surjective, as morphisms of graded $R$-modules.  In particular, if $R$ is a field, then $\Gamma$ is a relative Hopf algebra if it is injective.

For any algebra $E$ in $\VV$, and any relative Hopf algebra $\Gamma\colon  H\to K$,  let 
$$(A, \rho_{A})=(E\otimes H, E\otimes \Delta_{H})\quad\text{and}\quad (B, \rho_{B})=(E\otimes K, E\otimes \Delta_{K}).$$
The morphism $\big(H,(A, \rho_{A})\big) \xrightarrow{(\Gamma, E\otimes \Gamma)} \big(K,(B, \rho_{B})\big)$ in $\Comodalg_{\VV}$ is then a generalized Hopf-Galois extension, as $\vp^{\mathrm{co} \Gamma}$ is simply the identity on $E$, while 
$\Gal_{(\Gamma, E\otimes \Gamma)}$ is given by applying the functor $E\otimes-$ to the composite (\ref{eqn:relHopf}).   Following classical terminology, we call this morphism a \emph{normal relative Hopf-Galois extension} with \emph{normal basis} $E$.
\end{example}

\subsection{The homotopic Hopf-Galois framework}

\begin{convention}\label{conv:hhg} Henceforth $\VV$ denotes a symmetric monoidal model category satisfying Convention \ref{conv:m-t} and the CHF condition (Definition \ref{definition:CHF}). 
We also assume that $\Alg_{\VV}$ is equipped with a model category structure with weak equivalences created in $\VV$ and that the category $\Alg_{\VV}^{H}$ of $H$-comodule algebras with the model category structure right-induced from that of $\VV^{H}$ (the category of $H$-comodules in $\VV$, where we have forgotten the multiplicative structure on $H$) via the free-algebra/forgetful adjunction,  for any bialgebra $H$ that we consider.  It follows that 
$$\bigadjunction{\Alg_{\VV}^{H}}{\Alg_{\VV}^{K}}{\Gamma_{*}}{\Gamma^{*}}$$ 
is a Quillen adjunction for every morphism $\Gamma\colon   H\to K$ of bialgebras; see Remark \ref{rmk:QA-coring map}.   Explicit examples of such model category structures can be found in \cite{hess-shipley:retractive} and \cite{hkrs}.
\end{convention}

\begin{definition}  Let $A$ be an $H$-comodule algebra.  For any fibrant replacement $A^{f}$ of $A$ in $\Alg_{\VV}^{H}$, the algebra $(A^{f})^{\mathrm{co}H}$ is a \emph{model of the homotopy coinvariants} of the $H$-coaction on $A$, denoted (somewhat abusively) $A^{\mathrm{hco}H}$. 
\end{definition}

Given an object $(\Gamma, \vp)\colon   (H, A)\to (K, B)$ in $\Comodalg_{\VV}^{\to}$, we can construct an associated morphism of algebras $\vp^{\mathrm{hco}\Gamma}\colon   A^{\mathrm{hco}H}\to B^{\mathrm{hco}K}$ as follows, inspired by Remark \ref{remark:coinvariants}.  Let 
$$i_{B}\colon  B\xrightarrow \sim  B^{f}\quad \text{and} \quad i_{A}\colon  \Gamma_{*}A\xrightarrow \sim (\Gamma_{*}A)^{f}$$ 
be fibrant replacements in $\Alg_{\VV}^{K}$, and let $\vp^{f}\colon  (\Gamma_{*}A)^{f}\to B^{f}$ be an extension of $\vp$ to the fibrant replacements.   Since $\Gamma^{*}\colon \Alg_{\VV}^{K}\to \Alg_{\VV}^{H}$ is a right Quillen functor, 
$$\Gamma^{*}\big(\vp^{f}\big)\colon \Gamma^{*}\big((\Gamma_{*}A)^{f}\big)\to \Gamma^{*}\big(B^{f}\big)$$ 
is a morphism of fibrant $H$-comodule algebras.  

Let  $j\colon  A\xrightarrow \sim  A^{f}$ be any fibrant replacement  in $\Alg_{\VV}^{H}$. The composite morphism of $H$-comodule algebras
$$A\xrightarrow {\eta^{\Gamma}_{A}} \Gamma^{*}(\Gamma_{*}A) \xrightarrow{\Gamma^{*}(i_{A})} \Gamma^{*}\big((\Gamma_{*}A)^{f}\big)$$
extends to a morphism of $H$-comodule algebras 
$$\tilde \imath\colon   A^{f}\to  \Gamma^{*}\big((\Gamma_{*}A)^{f}\big),$$
since $j$ is an acyclic cofibration, and $ \Gamma^{*}\big((\Gamma_{*}A)^{f}\big)$ is fibrant.  A model for 
$$\vp^{\mathrm{hco}\Gamma}\colon   A^{\mathrm{hco}H}\to B^{\mathrm{hco}K}$$ 
is then given by the composite
\smallskip

\begin{equation}\label{eqn:phi-hcogamma}(A^{f})^{\mathrm{co}H}\xrightarrow{\tilde\imath^{\mathrm{co}H}}  \Big(\Gamma^{*}\big((\Gamma_{*}A)^{f}\big)\Big)^{\mathrm{co}H}\xrightarrow {\big(\Gamma^{*}(\vp^{f})\big)^{\mathrm{co}H}}\Big(\Gamma^{*}\big(B^{f}\big)\Big)^{\mathrm{co}H}\cong (B^{f})^{\mathrm{co}K}.
\end{equation}
\smallskip

To define homotopic relative Hopf-Galois extensions, we now modify somewhat  the approach  of \cite[Definition 3.2]{hess:hhg}, categorifying both conditions instead of just one. As we see below, under reasonable hypotheses a homotopic Hopf-Galois extension in the sense of  \cite[Definition 3.2]{hess:hhg} also satisfies the conditions of the modified definition below. 

\begin{definition}\label{definition:hhge}  A morphism  $(\Gamma, \vp)\colon   (H, A)\to (K, B)$ in $\Comodalg_{\VV}$ is a \emph{relative homotopic Hopf-Galois extension} if both of the adjunctions
$$\hugeadjunction{\VV_{A^{\mathrm{hco} H}}}{\VV_{B^{\mathrm{hco} K}}}{(\vp^{\mathrm{hco} \Gamma})_{*}}{(\vp^{\mathrm{hco} \Gamma})^{*}}$$
and
$$\hugeadjunction{\VV_{B}^{\Desc(\vp)}}{\VV_{B}^{\Hopf(\rho_{B}, \bar\mu_{K})}}{\Gal(\Gamma, \vp)_{*}}{\Gal(\Gamma, \vp)^{*}}$$
are Quillen equivalences.

A morphism $\Gamma\colon  H \to K$ of bialgebras in $\VV$ is a \emph{relative homotopic Hopf algebra} if $(\Gamma, \Gamma)\colon   (H,H)\to (K,K)$ is a relative homotopic Hopf-Galois extension.
\end{definition}

\begin{remark}  The definition of homotopic Hopf-Galois extension is independent of the choice of fibrant replacements for $A$ and $B$ underlying the definition of $A^{\mathrm{hco} H}$ and $B^{\mathrm{hco} K}$, since $\VV$ satisfies the CHF hypothesis, whence all weak equivalences of algebras are pure and therefore induce Quillen equivalences on module categories (Propositions \ref{prop:pure we alt} and \ref{prop:resext}).
\end{remark}

\begin{remark} In the special case of a morphism of the form $(R,A) \xrightarrow{(\eta_{H}, \vp)} (H,B)$ in $\Comodalg_{\VV}$, we recover a slightly modified version of the definition of a homotopic $H$-Hopf-Galois extension from \cite{hess:hhg}.
\end{remark}

\begin{remark}  In \cite{rognes}  Rognes defined homotopic Hopf-Galois extensions of commutative ring spectra in a convenient symmetric monoidal model category $\cat{S}$ of spectra, such as symmetric spectra and $S$-modules. According to his conventions, a morphism $(S,A) \xrightarrow{(\eta_{H}, \vp)} (H,B)$ in $\Comodalg_{\cat{S}}$, where $S$ is the sphere spectrum, and $A$ and $B$ are commutative $S$-algebras, is a homotopic Hopf-Galois extension if  the composite
$$A=A^{\mathrm{co} H} \xrightarrow{j ^{\mathrm{co} H}} A^{\mathrm{hco} H}\xrightarrow{\vp^{\mathrm{hco} \Gamma}} B^{\mathrm{hco} H},$$
where $j\colon A \xrightarrow \sim A^{f}$ is a fibrant replacement in $\Alg_{\cat S}^{H}$, and
$$\beta_{\vp}=\Gal_{(\eta_{H}, \vp)}\colon   \Desc(\vp) \to \Hopf(\rho_{B})$$
are weak equivalences, where $B^{\mathrm{hco} H}$ is modelled explicitly as the totalization of a certain cosimplicial ``cobar''-type construction. 

As it is still work in progress to show that all of conditions of Hypothesis \ref{conv:hhg} hold in various incarnations of $\cat S$ (cf.~\cite[Corollary 5.6]{hess-shipley:retractive}), we cannot yet apply the results below characterizing homotopic Hopf-Galois extensions to conclude that Rognes's definition fits precisely into our framework, but we strongly suspect that it is the case.
\end{remark}

\begin{remark} The generalization of homotopic Hopf-Galois extensions to a relative framework is not merely an idle exercise. Indeed, as shown in \cite{hess-karpova}, the formulation of one direction of a Hopf-Galois correspondence for Hopf-Galois extensions of differential graded algebras requires such relative extensions.
\end{remark}

As an immediate consequence of Proposition \ref{prop:resext} and Corollary \ref{cor:QE conditions}, we obtain conditions under which a morphism of comodule algebras is a relative homotopic Hopf-Galois extension.

\begin{proposition}\label{prop:hhg conditions} Let $\VV$ be a symmetric monoidal model category satisfying Convention \ref{conv:hhg}.   Let $(\Gamma, \vp)\colon   (H, A)\to (K, B)$ be a morphism in $\Comodalg_{\VV}$.

If $\vp^{\mathrm{hco} \Gamma}\colon   A^{\mathrm{hco} H} \to B^{\mathrm{hco} K}$ is a weak equivalence and  $\Gal(\Gamma, \vp)\colon   \Desc(\vp)\to \Hopf(\rho_{B}, \bar\mu_{K})$ is a copure weak equivalence, then $(\Gamma, \vp)$   is a relative homotopic Hopf-Galois extension.
\end{proposition}

\begin{corollary}\label{cor:htpic rel HA} Let $\VV$ be a symmetric monoidal model category satisfying Convention \ref{conv:hhg}. If the unit $\kk$ is fibrant, then a morphism $\Gamma\colon   H\to K$ of bialgebras in $\VV$ is a relative homotopic Hopf algebra if $\Gal (\Gamma, \Gamma)\colon   \Desc(\Gamma) \to \Hopf(\Delta_{K}, \bar\mu_{K})$ is a copure weak equivalence.
\end{corollary}

\begin{proof}  Since $\kk$ is fibrant in $\VV$, it is fibrant in $\Alg_{\VV}$, whence both $H$ and $K$ are fibrant in their respective categories of comodule algebras.  It follows that the identity on $\kk$ is a model of $\Gamma^{\mathrm{hco}\Gamma}\colon  H^{\mathrm{hco}H}\to K^{\mathrm{hco}K}$.
\end{proof}

\section{Homotopic Hopf-Galois extensions of chain algebras}\label{sec:ch cx}
In this section we illustrate the theory of the previous section when the underlying monoidal model category is that of unbounded chain complexes over a commutative ring $R$, endowed with the usual monoidal structure and the \emph{Hurewicz model structure} \cite{barthel-may-riehl}, in which the weak equivalences are the chain homotopy equivalences, the fibrations are the degreewise-split surjections, and the cofibrations are the degreewise-split injections.  In particular we provide a large class of examples of homotopic Hopf-Galois extensions and  prove a theorem analogous to the descent-type description of homotopic Hopf-Galois extensions in  \cite[Proposition 12.1.8]{rognes}.  The work in this section builds on \cite[Section 5]{berglund-hess:morita}, the key results of which we recall below.

\subsection{Homotopy theory of chain modules and comodules}
Let $A$ be an algebra in $\Ch_{R}$. As shown in \cite[Theorems 4.5, 4.6, and 6.12]{barthel-may-riehl},  the category $(\Ch_{R})_{A}$ admits a proper, monoidal model category structure right-induced from the Hurewicz structure on $\Ch_{R}$ by the adjunction
$$\bigadjunction{\Ch_{R}}{(\Ch_{R})_{A},}{-\tensor A}{\UU}$$ 
which we call the \emph{relative model structure}. A morphism of $A$-modules is thus a weak equivalence (respectively, fibration) in the relative structure if the underlying morphism of chain complexes is a chain homotopy equivalence  (respectively, a degreewise-split surjection). We call the distinguished classes with respect to the relative model structure \emph{relative weak equivalences}, \emph{relative fibrations}, and \emph{relative cofibrations}, and $A$-modules that are cofibrant with respect to the relative model structure are called \emph{relative cofibrant}. The category of left modules admits an analogous relative structure.

Barthel, May, and Riehl provided the following useful characterization of relative cofibrant objects in  ${}_{A}(\Ch_{R})$.  A similar result holds for $(\Ch_{R})_{A}$.

\begin{proposition}\cite[Theorem 9.20]{barthel-may-riehl}\label{prop:bmr-cofib} An object $M$ in ${}_{A}(\Ch_{R})$ is relative cofibrant  if and only if it is a retract of an $A$-module $N$ that admits a filtration
$$0=F_{-1}N\subseteq F_{0}N\subseteq \cdots \subseteq F_{n}N\subseteq F_{n+1}N \subseteq \cdots$$
where $N=\bigcup_{n\geq 0}F_{n}N$ and for each $n\geq 0$, there is chain complex $X(n)$ with 0 differential such that $F_{n}N/F_{n-1}N\cong A \otimes X(n)$.
\end{proposition}

Barthel, May, and Riehl call  filtrations of this sort \emph{cellularly $r$-split} and show that the  inclusion maps $F_{n-1}N\to F_{n}N$ are split as nondifferential, graded $A$-modules (cf.~\cite[Definition 9.17]{barthel-may-riehl}). Note in particular that $A$ itself is always cofibrant as a right or left $A$-module.

\begin{remark}\label{rmk:cofib-tens}  It follows from Proposition \ref{prop:bmr-cofib} that if an object $M$ in ${}_{A}(\Ch_{R})$ is relative cofibrant, then $M\otimes Y$ is also relative cofibrant, for every chain complex $Y$.  Indeed, if $M$ is a retract of a left $A$-module $N$ with a cellularly $r$-split filtration 
$$0=F_{-1}N\subseteq F_{0}N\subseteq \cdots \subseteq F_{n}N\subseteq F_{n+1}N \subseteq \cdots,$$
then $M\otimes Y$ is a retract of the left $A$-module $N\otimes Y$, which has a (not necessarily cellularly) $r$-split filtration 
$$0=(F_{-1}N)\otimes Y\subseteq (F_{0}N)\otimes Y\subseteq \cdots \subseteq (F_{n}N)\otimes Y\subseteq (F_{n+1}N)\otimes Y \subseteq \cdots,$$
and thus also a cellularly $r$-split filtration by \cite[Theorem 9.20]{barthel-may-riehl}.
\end{remark}

Specializing the definition of cellularly $r$-split filtrations somewhat, we obtain an important class of relative cofibrant modules.

\begin{definition}  An object $M$ in ${}_{A}(\Ch_{R})$ is \emph{flat-cofibrant} with respect to the relative model structure if it is a retract of an $A$-module $N$ that admits a cellularly $r$-split filtration
$$0=F_{-1}N\subseteq F_{0}N\subseteq \cdots \subseteq F_{n}N\subseteq F_{n+1}N \subseteq \cdots$$
with  $F_{n}N/F_{n-1}N\cong A \otimes X(n)$ where $X(n)$ is degreewise $R$-flat, which we call a \emph{cellularly $r$-split flat filtration}.
\end{definition}

\begin{proposition}\label{prop:ch-special-modules}\cite[Proposition 5.7]{berglund-hess:morita}  Let $A$ be an algebra in $\Ch_{R}$, and let $N$ be a left $A$-module.
\begin{enumerate}
\item If $N$ is cofibrant in the relative structure on ${}_{A}(\Ch_{R})$, then it is is homotopy flat and homotopy projective.    In particular, the category ${}_{A}(\Ch_{R})$ satisfies the CHF hypothesis (Definition \ref{definition:CHF}). 
\item If $N$ is flat-cofibrant in the relative structure on ${}_{A}(\Ch_{R})$, then it is flat and therefore strongly homotopy flat.
\item If  $N$ contains $A$ as a summand, then $N$ is homotopy faithful and homotopy cofaithful.  
\item Every algebra morphism $A\to B$ with underlying chain homotopy equivalence is homotopy pure.
\end{enumerate}
\end{proposition}

Together with condition (1) of the proposition above, Proposition 2.11 immediately implies the following result.

\begin{proposition}\label{prop:resext-chain} Let $\vp\colon A\to B$ be a morphism of algebras in $\Ch_{R}$. The induced restriction/extension-of-scalars adjunction 
$$\bigadjunction{\VV_A}{\VV_B}{\vp_{*}}{\varphi^*},$$
is a Quillen equivalence with respective to the relative model structures if and only if the chain map underlying $\varphi\colon A\rightarrow B$ is a chain homotopy equivalence.
\end{proposition}

The existence of model category structure for categories of comodules over corings in the context of unbounded chain complexes was proved in \cite{hkrs}.

\begin{theorem}\cite[Theorem 6.6.3]{hkrs} \label{thm:hkrs} Let $R$ be any commutative ring. For any algebra $A$ in $\Ch_{R}$ and any $A$-coring $C$, the category $(\Ch_{R})_A^C$ of $C$-comodules in $A$-modules admits a model category structure left-induced from the relative model structure  on $(\Ch_{R})_{A}$ via the forgetful functor.   
\end{theorem}

\begin{remark}\label{rmk:limits-comods} Note that if $(A,C)$ is a flat coring, e.g., if $C$ is flat-cofibrant as a left $A$-module, then limits in $(\Ch_{R})_{A}^{C}$ are in fact created in $(\Ch_{R})_{A}$ and thus in $\Ch_{R}$.
\end{remark}

In \cite{berglund-hess:morita} the authors established the existence of interesting classes of copure weak equivalences of corings and of algebra morphisms satisfying effective homotopic descent in the chain complex framework.

\begin{theorem} \label{thm:copure}\cite[Theorem 5.16]{berglund-hess:morita} Let $A$ be an algebra in $\Ch_{R}$.  If $C$ is a flat $A$-coring, and $D$ is a coaugmented flat-cofibrant $A$-coring, then every relative weak equivalence $f\colon C \to D$ of $A$-corings  is copure.
\end{theorem}

\begin{theorem} \label{thm:homotopic descent for chains}\cite[Theorem 5.17]{berglund-hess:morita} Let $\vp:A\to B$ be algebras in $\Ch_{R}$. If as a left $A$-module $B$ is flat-cofibrant and contains $A$ as a retract, then $\vp$ satisfies effective homotopic descent.
\end{theorem}

Finally, putting all of the pieces together, we can describe a class of morphisms of corings that induce Quillen equivalences between the corresponding comodule categories.

\begin{theorem} \label{thm:we=QE for dg-corings}\cite[Theorem 5.18]{berglund-hess:morita}
Let $(\varphi,f)\colon (A,C)\rightarrow (B,D)$ be a morphism of flat corings in $\Ch_{R}$ such that, as a left $A$-module, $B$ is flat-cofibrant and contains $A$ as a retract and such that $D$ is coaugmented and flat-cofibrant as a left $B$-module.

The adjunction governed by $(\varphi,f)$,
$$\bigadjunction{(\Ch_{R})_A^C}{(\Ch_{R})_B^D}{(\vp, f)_{*}}{(\vp, f)^{*}},$$
is a Quillen equivalence if and only if the morphism of $B$-corings $f\colon B_*(C) \rightarrow D$ is a relative weak equivalence.
\end{theorem}

As explained in Remark 5.20 in \cite{berglund-hess:morita}, the flat-cofibrancy hypothesis on $\vp$ in the theorem above is not too restrictive, since $\vp$ can always be replaced up to weak equivalence by a morphism that satisfies it.

\subsection{Homotopy theory of chain comodule algebras}
Before discussing homotopic Hopf-Galois extensions in $\Ch_{R}$, it remains to obtain a model category structure on $\Alg_R^H$, the category of the $H$-comodule algebras in $\Ch_{R}$, for any bialgebra $H$.  The first step towards such a theorem is provided by the following result, a special case of the model category structure on categories of comodules over corings. 

\begin{proposition}\label{prop:comodule}\cite[Corollary 6.3.7]{hkrs} Let $R$ be any commutative ring and $C$ any coalgebra in $\Ch_{R}$.
There exists a  model category structure on the category $\Ch_R^C$ of right $C$-comodules that is left-induced  along the forgetful functor $\Ch_R^C \to \Ch_{R}$, with respect to the Hurewicz model structure on $\Ch_{R}$.
\end{proposition}  

In particular, if $H$ is a bialgebra in $\Ch_{R}$, one can forget its multiplicative structure and apply the proposition above to obtain a model category structure on $\Ch_{R}^{H}$, the category of comodules over the coalgebra underlying $H$.

\begin{theorem}\label{thm:comod-alg}\cite[Theorem 6.5.1]{hkrs} Let $R$ be any commutative ring and $H$ any bialgebra in $\Ch_{R}$.There exists  a right-induced model structure on $\Alg_R^H$, created by the forgetful functor $\Alg_R^H \to \Ch_{R}^{H}$, with respect to the model structure on $\Ch_R^H$ of Proposition \ref{prop:comodule}. 
\end{theorem} 

It is important for our study of homotopic Hopf-Galois extensions to know that for any bialgebra $H$ in $\Ch_{R}$, the two-sided cobar construction provides a canonical fibrant replacement functor
$$\Om(-;H;H)\colon \Alg_R^H\to  \Alg_R^H: A\mapsto \Om (A;H;H),$$
and a natural relative weak equivalence of $H$-comodule algebras  $\iota_{A}\colon   A \to \Om(A;H;H)$. We establish this result as follows.

We recall first the well known definition of the cobar construction for comodules over a coaugmented, differential graded coalgebra.  

\begin{notation} Let $T$ denote the free tensor algebra functor, which to any graded $R$-module $V$ associates the graded $R$-algebra $TV=\kk \oplus \bigoplus _{n\geq 1} V^{\otimes n}$, the homogeneous elements of which are denoted $v_{1}|\cdots |v_{n}$.

For any graded $R$-module $V$, we let $\si V$ denote the graded $R$-module with $\si V_{n}\cong V_{n+1}$ for all $n$, where the element of $\si V_{n}$ corresponding to $v\in V_{n+1}$ is denoted $\si v$.

Let  $(C,\Delta, \ve, \eta)$ be a coaugmented coalgebra in $\Ch_{R}$, with coaugmentation coideal $\overline C=\operatorname{coker} (\eta\colon   R \to C)$.  We use the Einstein summation convention and  write $\Delta (c)= c_{i}\otimes c^{i}$ for all $c\in C$ and similarly for the map induced by $\Delta$ on  $\overline C$.  If $(M, \rho)$ is a right $C$-comodule, then we apply the same convention again and write $\rho(x)=x_{i}\otimes c^{i}$ for all $x\in M$, and similarly for a left $C$-comodule.
\end{notation}

\begin{definition} Let  $(C,\Delta, \ve, \eta)$ be a coaugmented coalgebra in $\Ch_{R}$, with coaugmentation coideal $\overline C=\operatorname{coker} (\eta\colon   \kk \to C)$.  For any right $C$-comodule $(M,\rho)$  and left $C$-comodule $(N, \lambda)$, let $\Om(M;C;N)$ denote the object in $\Ch_{R}$
$$(M\otimes T (s^{-1}  \overline C) \otimes C, d_{\Om}),$$
where
\begin{align*}
d_{\Om}(x \otimes \si c_{1}|\cdots |\si c_{n}\otimes y)= &\,dx \otimes \si c_{1}|\cdots |\si c_{n}\otimes y\\
&+x\otimes \sum _{j=1}^{n}\pm\si c_{1}|\cdots |\si dc_{j}|\cdots \si c_{n}\otimes y\\
&\pm x \otimes \si c_{1}|\cdots |\si c_{n}\otimes dy\\
&\pm x_{i}\otimes \si c^{i}| \si c_{1}|\cdots |\si c_{n}\otimes y\\
&+x\otimes \sum _{j=1}^{n}\pm\si c_{1}|\cdots |\si c_{j,i}|\si c_{j}^{{i}}|\cdots \si c_{n}\otimes y\\
&\pm x\otimes \si c_{1}|\cdots |\si c_{n}|\si c_{i}\otimes y^{i}
\end{align*}
where all signs are determined by the Koszul rule, the differentials of $M$, $N$, and $C$ are all denoted $d$, and $\si 1=0$ by convention. 

 If  $N=C$, then $\Om (M;C;C)$ admits a right $C$-comodule structure induced from the rightmost copy of $C$.  
\end{definition}

\begin{remark}\label{rmk:coinv} The cobar construction $\Om (M;C;C)$ is a ``cofree resolution'' of $M$, in the sense that the coaction map $\rho\colon M\to M\otimes C$ factors in $\Ch_{R}^{C}$ as
\begin{equation}\label{eqn:factor-coaction}\xymatrix{ M\ar [dr]_{\widetilde \rho}\ar [rr]^{\rho}&&M\otimes C,\\ 
&\Om(M;C;C)\ar [ur]_{q}}
\end{equation}
where $\widetilde \rho (x)=x_{i}\otimes 1\otimes c^{{i}}$  and $q(x\otimes 1\otimes c)=x\otimes c$, while $q(x \otimes \si c_{1}|\cdots |\si c_{n}\otimes c)=0$ for all $n\geq 1$.   It is well known that the composite
$$\Om (M;C;C) \xrightarrow q M\otimes C \xrightarrow {M\otimes \ve} M$$
is a chain homotopy equivalence, by a standard ``extra degeneracy'' argument. It follows that $\tilde \rho\colon  M\to \Om (M;C;C)$ is always a relative weak equivalence of right $C$-comodules. Moreover, $\Om (M;C;C)^{\mathrm{co} C}\cong \Om(M;C;R)$.  
\end{remark}

When $C$ is replaced by a bialgebra and $M$  by  a comodule algebra, then the cobar construction admits a compatible multiplicative structure, building on the following result from \cite{hess-levi}.

\begin{lemma} \cite[Corollary 3.6]{hess-levi}  If $H$ is a bialgebra  in $\Ch_{R}$, the cobar construction lifts to a functor 
$$\Om(-;H;R)\colon  \Alg_{R}^{H}\to \Alg_{R}$$
where  for every $H$-comodule $A$
\begin{align*}(a\otimes 1)(a'\otimes 1)&=aa'\otimes 1, \quad\forall\; a,a'\in A;\\
(a \otimes w)(1 \otimes w')&=a\otimes ww', \quad\forall\; a\in A, w,w'\in \Om H;\\
(1\otimes \si h)(a\otimes 1)&= (-1)^{(\deg h+1)\deg a_{i}}a_{i}\otimes \si (h a^{i}), \quad\forall\; a\in A, h\in H.
\end{align*}
\end{lemma}

An analogous formula holds for the multiplicative structure on $\Om(R;H;A)$, if $A$ is a left $H$-comodule algebra.

As Karpova showed in \cite{karpova}, the multiplication defined above extends to $\Om (A;H;H)$.

\begin{lemma}\cite[Section 2.1.3]{karpova}\label{lem:factor-alg} If $H$ is a bialgebra  in $\Ch_{R}$, the cobar construction lifts to a functor 
$$\Om(-;H;H)\colon  \Alg_R^{H}\to \Alg_R^{H}$$
such that for any $H$-comodule algebra $(A,\rho)$,  the coaction map $\rho\colon   A\to A\otimes H$ factors in $\Alg_R^{H}$ as
\begin{equation}\label{eqn:factor-coaction2}\xymatrix{ A\ar [dr]_{\widetilde \rho}\ar [rr]^{\rho}&&A\otimes H,\\ 
			&\Om(A;H;H)\ar [ur]_{q}}
			\end{equation}
where $\widetilde \rho (a)=a_{i}\otimes 1\otimes h^{{i}}$  and $q(a\otimes 1\otimes h)=a\otimes h$, while $q(a \otimes \si h_{1}|\cdots |\si h_{n}\otimes h)=0$ for all $n\geq 1$.   The multiplication on $\Om (A;H;H)$ is determined by the multiplication in $\Om (A;H;R)$ and $\Om(R;H;A)$, together with the formulas
\begin{align*}
(1 \otimes s^{-1}h_1 \otimes h')(a \otimes s^{-1}h_2 \otimes 1)&= \big( (1 \otimes s^{-1}h_1)(a \otimes 1) \otimes 1\big)\big( 1 \otimes (1 \otimes h')(s^{-1}h_2 \otimes 1)\big),\\
(a \otimes 1 \otimes 1)(a' \otimes s^{-1}h_1 \otimes h') &= aa' \otimes s^{-1}h_1 \otimes h',\\
(a \otimes s^{-1}h_1 \otimes h')(1 \otimes 1 \otimes h'' )& =  a \otimes s^{-1}h_1 \otimes h'h'', 
\end{align*}
for all $a,a' \in A$, $h,h',h_1,h_2 \in H$.
\end{lemma}

Under reasonable conditions on $H$, the two-sided cobar construction $\Om (A;H;H)$ has particularly nice properties as a left $A$-module.

\begin{proposition}\label{prop:om(a,h,h)} If $H$ is a bialgebra  in $\Ch_R$ that is degreewise $R$-flat, and $A$ is an $H$-comodule algebra, then $\Om (A;H;H)$ is homotopy faithfully flat as a left $A$-module, with respect to the structure induced by the algebra map $\tilde \rho\colon   A \to \Om (A;H;H)$. 
\end{proposition}

\begin{proof}  To see that $\Om (A;H;H)$ is homotopy flat, observe first that for any right $A$-module $M$, the graded $R$-module underlying $M\otimes _{A}\Om (A;H;H)$ is isomorphic to $M\otimes T(\si \overline H)\otimes H$.  Just as in the proofs of \cite[Theorem 7.8]{hess-shipley} and \cite[Theorem 5.15]{berglund-hess:morita}, we can construct $M\otimes _{A}\Om (A;H;H)$ as the limit in $\Ch_{R}^{H}$ of a tower
$$ ...\xrightarrow {q^{n+1}} E_{n}M\xrightarrow {q^{n}} E_{n-1}M\xrightarrow {q_{n-1}} ...\xrightarrow{q_{1}} E_{0}M=M\otimes_{} H$$
 in $(\Ch_{R})^H$, natural in $M$, where each morphism $q_{n}:E_{n}M\to E_{n-1}M$ is given by a pullback  in $(\Ch_{R})^H$ of the form
 $$\xymatrix{E_{n}M\ar [d]_{q_{n}}\ar [r]& \operatorname{Path}(B_{n}M)\otimes H\ar[d]_{p_{n}\otimes H}\\
 E_{n-1}M\ar [r]^{k_{n}}& B_{n}M\otimes H,}$$
 where $B_{n}$ is an functor from $(\Ch_{R})_{A}$ to $\Ch_{R}$, and $p_{n}:\operatorname{Path}(B_{n}M) \to B_{n}M$ is the natural map from the contractible path-object on $B_{n}M$ to  $B_{n}M$ itself, which is a relative fibration of right $A$-modules \cite[Lemma 5.5]{berglund-hess:morita}.  It is important here that $H$ be flat over $A$, so that pullbacks in $(\Ch_{R})^H$ are created in $\Ch_{R}$. 
 
 Given the natural decomposition of $M\otimes _{A}\Om (A;H;H)$ as the limit of a tower of fibrations, an inductive proof, very similar to that of \cite[Theorem 5.15]{berglund-hess:morita}, enables us to show that if $f:M\to N$ is a relative weak equivalence, then $$f\otimes_{A}\Om (A;H;H): M\otimes _{A}\Om (A;H;H) \to  N\otimes _{A}\Om (A;H;H)$$ is a chain homotopy equivalence, i.e., $\Om (A;H;H)$ is homotopy flat.

Since the nondifferential graded $A$-module underlying $\Om(A;H;H)$ is $A$-free, $-\otimes _{A}\Om (A;H;H)$ preserves kernels.  The functor $-\otimes _{A}\Om (A;H;H)$ also preserves all finite products, since they are isomorphic to finite sums.  It follows that $-\otimes_{A}\Om (A;H;H)$ preserves all finite limits, whence $\Om (A;H;H)$ is homotopy faithfully flat.

Since $A$ is a retract of $\Om(A;H;H)$ as an algebra and therefore as an $A$-module, $\Om (A;H;H)$ is homotopy faithful.
\end{proof}

A proof essentially identical to that of \cite[Theorem 7.8]{hess-shipley} enables us to establish the following result; see also \cite[Theorem 5.15]{berglund-hess:morita}.  

\begin{theorem}\label{thm:fib-repl} Let $C$ be a coalgebra in $\Ch_R$ that is degreewise $R$-flat.
For every $C$-comodule $(M,\rho)$,  the maps $\widetilde \rho$ and  $q$ of diagram (\ref{eqn:factor-coaction2}) are a trivial cofibration and a fibration, respectively, in the model structure on  $\Ch_{R}^C$ of Proposition \ref{prop:comodule}.    Moreover, both the source and the target of $q$ are fibrant in $\Ch_R^{C}$, whence $\Om(M;C;C)$ is a fibrant replacement of $M$ in $\Ch_{R}^{C}$.
\end{theorem}

Since the fibrations and weak equivalences in $\Alg_{R}^{H}$ are created in $\Ch_{R}^{H}$,  the next result is an immediate consequence of Lemma \ref{lem:factor-alg} and Theorem \ref{thm:fib-repl}.

\begin{corollary}\label{cor:fib-repl} Let $H$ be a bialgebra in $\Ch_R$ that is degreewise $R$-flat.
For every $H$-comodule algebra $(A,\rho)$,  the maps $\widetilde \rho: A \to \Om (A;H;H)$ and  $q:\Om (A;H;H)\to A\otimes H$ are a trivial cofibration and a fibration, respectively, in  $\Alg_R^H$.    Moreover, both the source and the target of $q$ are fibrant in $\Alg_R^{H}$, whence $\Om(A;H;H)$ is a fibrant replacement of $A$ in $\Alg_R^{H}$.
\end{corollary}

The second part of Remark \ref{rmk:coinv} implies that the next result is an immediate consequence of Corollary \ref{cor:fib-repl}.

\begin{corollary}\label{cor:hcoH-model}  Let $H$ be a bialgebra in $\Ch_R$ that is degreewise $R$-flat.
For every $H$-comodule algebra $(A,\rho)$,  the algebra $\Om (A;H;R)$ is a model of $A^{\mathrm{hco}H}$.
\end{corollary}

Before illustrateing further the utility of Corollary \ref{cor:fib-repl}, we need to define a special condition on the coalgebra structure of the bialgebras we study.

\begin{definition}\cite[Definition 6.4.6]{hkrs} A coaugmented coalgebra $C$ in $\Ch_{R}$ is \emph{split-conilpotent} if there is a sequence 
$$R=C[-1]\xrightarrow {j_{0}}C[0]\xrightarrow {j_{1}}\cdots \xrightarrow{j_{n-1}} C[n-1]\xrightarrow {j_{n}} C[n]\xrightarrow {j_{n+1}} \cdots $$
of degreewise-split inclusions of subcoalgebras of $C$ such that  $C=\operatorname{colim}_{n}C_{n}$ and $\coker j_{n}$ is a trivial (non-counital) coalgebra for all $n$.
\end{definition}

\begin{remark}\label{rmk:split} As observed in \cite[Remark 6.4.7]{hkrs}, over a field any conilpotent coalgebra is split-conilpotent, while over an arbitrary commutative ring $R$, any coalgebra in $\Ch_{R}$ with cofree underlying graded coalgebra is split-conilpotent.  In particular, for any augumented algebra $A$ in $\Ch_{R}$, its bar construction $\Bar A$ is a split-conilpotent coalgebra.  On the other hand, by \cite[Corollary 6.4.3]{hkrs}, if $H$ is a conilpotent bialgebra (i.e., as a coalgebra, it is the colimit of its primitive filtration), then the counit $\ve_{H}:H \to \Bar \Om H$ of the cobar-bar adjunction is a relative weak equivalence of bialgebras, i.e., the underlying map of chain complexes is a chain homotopy equivalence.  In other words, every conilpotent bialgebra is relative weakly equivalent to a split-conilpotent bialgebra.  Moreover, if $H$ is degreewise $R$-flat, then $\ve_{H}$ induces a Quillen equivalence
$$\adjunction{\Alg_{R}^{H}}{\Alg_{R}^{\Bar\Om H},}{(\ve_{H})_{*}}{(\ve_{H})^{*}}$$
since \cite[Theorem 5.16]{berglund-hess:morita} implies that $\ve_{H}$ is copure.  For homotopy-theoretic purposes, there is no loss of generality, therefore, in assuming that any degreewise $R$-flat, conilpotent bialgebra is in fact split-conilpotent.
\end{remark}

The main reason for our interest in split-conilpotent bialgebras lies in the next result.

\begin{proposition}\label{prop:hyp1} For any bialgebra $H$ in $\Ch_R$ that is degreewise $R$-flat and split-conilpotent as a coalgebra, the functor $$(-)^{\mathrm{co}H}\colon   \Alg_R^{H}\to \Alg_R$$ reflects weak equivalences between fibrant objects.
\end{proposition}

A key step in the proof of this proposition requires the following lemma from \cite{berglund-hess:morita}.

\begin{lemma}[The Homotopy Five Lemma]\label{lem:h5}\cite[Lemma 5.1]{berglund-hess:morita} Let
$$\xymatrix{0\ar[r]&M'\ar [r]\ar[d]_{f'}&M\ar [r]\ar[d]_{f}&M''\ar [r]\ar[d]_{f''}&0\\0\ar[r]&N'\ar [r]&N\ar [r]&N''\ar [r]&0}$$
be a commuting diagram in $\Ch_{R}$, where the rows are degreewise-split, exact sequences. If $f'$ and $f''$ are chain homotopy equivalences, then so is $f$.
\end{lemma}

\begin{proof}[Proof of Proposition \ref{prop:hyp1}]  Let $\vp\colon   (A,\rho_{A}) \to (B,\rho_{B})$ be a morphism of fibrant $H$-comodule algebras such that $\vp^{\mathrm{co}H}$ is a relative weak equivalence.  To show that $\vp$ is necessarily also a relative weak equivalence, first consider the commutative diagram of fibrant $H$-comodule algebras
$$\xymatrix{A \ar [d]_{\tilde \rho_{A}}^{\sim }\ar [rr]^{\vp}&& B\ar[d]_{\tilde \rho_{B}}^{\sim }\\
\Om (A;H;H)\ar[rr]^{\Om(\vp;H;H)}&& \Om (B;H;H).}$$
Applying $(-)^{\mathrm{co}H}$, we obtain a commutative diagram of algebras
$$\xymatrix{A^{\mathrm{co}H} \ar [d]_{\tilde \rho_{A}^{\mathrm{co}H}}^{\sim }\ar [rr]^{\vp^{\mathrm{co}H}}_{\sim }&& B^{\mathrm{co}H}\ar[d]_{\tilde\rho_{B}^{\mathrm{co}H}}^{\sim }\\
\Om (A;H;R )\ar[rr]^{\Om(\vp;H;R )}&& \Om (B;H;R ),}$$
where the vertical arrows are still relative weak equivalences, since $(-)^{\mathrm{co}H}$ is a right Quillen functor, and the top horizontal arrow is a relative weak equivalence by hypothesis.  By two-out-of-three, $\Om (\vp;H;R)$ is also a relative weak equivalence.

We can now prove by induction, using the Homotopy Five Lemma, that $\Om(\vp; H;H)$ is also a relative weak equivalence.   Let
$$R=H[-1]\xrightarrow {j_{0}}H[0]\xrightarrow {j_{1}}\cdots \xrightarrow{j_{n-1}} H[n-1]\xrightarrow {j_{n}} H[n]\xrightarrow {j_{n+1}} \cdots $$
be a sequence of degreewise-split inclusions of subcoalgebras of $H$ such that  $H=\operatorname{colim}_{n}H_{n}$ as colgebras, and $\coker j_{n}$ is a trivial (non-counital) coalgebra for all $n$, which induces
induces a filtration of $\Om (A;H;H)$ as chain complexes
$$\Om (A;H;R)\subseteq \Om \big(A;H;H[0]\big) \subseteq \cdots \subseteq \Om \big(A;H;H[n-1]\big) \subseteq \Om \big(A;H;H[n]\big) \subseteq\cdots  ,$$
with  filtration quotients 
$$\Om \big(A;H;H[n]\big)/\Om \big(A;H;H[n-1]\big)\cong\Om (A;H;R)\otimes \coker j_{n},$$
on which the differential is of the form $d_{\Om}\otimes 1 + 1 \otimes \bar d$, where $\bar d$ the induced differential on $\coker j_{n}$.   There is, of course, an analogous filtration of $\Om (B;H;H)$.   

We showed above that $\vp$ induces a relative weak equivalence, $\Om (\vp;H;R)$, from the 0th filtration stage of $\Om(A;H;H)$ to the 0th filtration stage of $\Om (B;H;H)$.  Suppose now that $\Om (\vp; H;P_{n}H): \Om \big(A;H;H[n-1]\big) \to \Om \big(B;H;H[n-1]\big)$ is a relative weak equivalence for some $n\geq 1$.  Consider the commuting diagram of  degreewise-split short exact sequences
$$\xymatrix{0\ar[r]&\Om \big(A;H;H[n-1]\big)\ar [r]\ar[d]_{\Om \big(\vp;H;H[n-1]\big)}&\Om \big(A;H;H[n]\big)\ar [r]\ar[d]_{\Om \big(\vp;H;H[n]\big)}&\Om (A;H;R)\otimes \coker j_{n}\ar [r]\ar[d]_{\Om (\vp; H; R)\otimes 1}&0\\0\ar[r]&\Om \big(B;H;H[n-1]\big)\ar [r]&\Om \big(B;H;H[n]\big)\ar [r]&\Om (B;H;R)\otimes \coker j_{n}\ar [r]&0.}$$
The leftmost and rightmost vertical arrows are chain homotopy equivalences by hypothesis and therefore, by the Homotopy Five Lemma, the middle vertical arrow is as well.  It follows that $\Om \big(\vp;H;H[n]\big):\Om \big(A;H;H[n]\big) \to \Om \big(B;H;H[n]\big)$ is a relative weak equivalence for all $n$ and thus that $\Om(\vp; H;H)$ is a relative weak equivalence, as the filtrations of $\Om(A;H;H)$ and $\Om(B;H;H)$ can be seen as colimits of directed systems of cofibrations of cofibrant objects in $\Ch_{R}$.

Finally, two-out-of-three applied to the first diagram implies that $\vp$ is a relative weak equivalence as well.
\end{proof}

\subsection{Homotopic relative Hopf-Galois extensions of chain algebras}

We can now provide concrete examples of homotopic relative Hopf-Galois extensions in $\Ch_R$, as well as conditions under which being a homotopic Hopf-Galois extension is equivalent to satisfying homotopic descent, which enables us moreover to include a generalized notion of Koszul duality in our global picture.

We begin by establishing the existence of a useful class of homotopic relative Hopf algebras.

\begin{lemma} A morphism of bialgebras $\Gamma\colon  H\to K$ in $\Ch_R$ is a homotopic relative Hopf algebra if $H$ and $K$ are degreewise $R$-flat, and $K$ is cofibrant as a left $H$-module and contains $H$ as a summand, with respect to the structure induced by $\Gamma$.
\end{lemma}

\begin{proof} By Corollary \ref{cor:htpic rel HA}, $\Gamma$ is a homotopic relative Hopf algebra if 
$$\Gal (\Gamma, \Gamma)\colon   \Desc(\Gamma) \to \Hopf(\Delta_{K}, \bar\mu_{K})$$ 
is a copure weak equivalence, since $R$ is fibrant in $\Ch_R$.  On the other hand, since $K$ is $H$-cofibrant, the left $H$-modules underlying $\Desc(\Gamma)$ and $\Hopf(\Delta_{K}, \bar\mu_{K})$, which are $K\otimes _{H}K$ and $K\otimes \Cof (\Gamma)$, are also cofibrant by Remark \ref{rmk:cofib-tens}.   

Theorem \ref{thm:copure} in the case $A=B=R$ implies that it suffices therefore to prove that $\Gal(\Gamma, \Gamma)$ is a relative weak equivalence.  It follows from the discussion of relative Hopf algebras in Example \ref{example:normal} that if $K$  admits a cellularly $r$-split filtration as an $H$-module, then  $\Gal(\Gamma, \Gamma)$ is actually an isomorphism by \cite[Theorem 4.4]{milnor-moore}.  Since retracts of isomorphisms are isomorphisms, it follows that $\Gal(\Gamma, \Gamma)$ is an isomorphism whenever $K$ is relative cofibrant as a left $H$-module.
\end{proof}

For any nice enough homotopic relative Hopf algebra, one can construct  homotopy-theoretic analogues of the ``normal basis'' extension in Example \ref{example:normal}.  Given a morphism of bialgebras $\Gamma:H\to K$, if $K$ admits a cellularly $r$-split filtration
$$0=F_{-1}K\subseteq F_{0}K\subseteq \cdots \subseteq F_{n}K\subseteq F_{n+1}K \subseteq \cdots$$
as left $H$-modules, notice that each $F_{n}K$ is also a left $K$-comodule and that the colimit respects the left $K$-comodule structure.

\begin{proposition}\label{prop:hyp2} Let $H$ and $K$ be degreewise $R$-flat bialgebras in $\Ch_{R}$. Let $\Gamma\colon   H\to K$ be a homotopic relative Hopf algebra  in $\Ch_R$ such that as a left $H$-module, $K$ admits a cellularly $r$-split filtration  and such that $\Gamma$ admits a retraction $K\to H$ that is a morphism of $H$-modules and $K$-comodules. Let $E$ be a $K$-comodule algebra in $\Ch_{R}$ that is degreewise $R$-flat. 

If $A=\Om(E;K;H)$, $B=\Om(E;K;K)$,  and $\vp=\Om (E;K; \Gamma)$, then 
$$(\Gamma, \vp)\colon  (H,A) \to (K,B)$$
is a homotopic relative  Hopf-Galois extension, and $\vp$ satisfies effective homotopic descent.

In particular, for any degreewise $R$-flat Hopf algebra $K$ in $\Ch_{R}$,  
$$\big(\Gamma, \Om (E;K;\eta)\big): \big(R,\Om(E;K;R)\big)\to \big(K, \Om (E;K;K)\big)$$
is a homotopic relative Hopf-Galois extension, and $\Om (E;K;\eta)$ satisfies homotopic descent, where $\eta:R\to K$ denotes the unit of $K$.
\end{proposition}  

\begin{proof} It is clear that $(\Gamma, \vp)$  is a morphism in $\Comodalg_{\Ch_{R}}$. Since $B$ is a fibrant object in $\Alg_{\Ch_{R}}^{K}$, and $\Gamma^{*}$ is a right Quillen functor, $A$ is a fibrant object in $\Alg_R^{H}$, as $A=\Gamma^{*}B$.  It follows that a model of
$$\vp^{\mathrm{hco}\Gamma}\colon   A^{\mathrm{hco}H}\to B^{\mathrm{hco}K}$$
is the identity on $ \Om(E; K;R)$, whence the first adjunction in Definition \ref{definition:hhge} is an actual equivalence of categories.  Moreover, since 
$$B\otimes _{A}B\cong \Om (E;K;K\otimes_{H}K)\quad\text{and}\quad B\otimes \Cof(\Gamma)\cong \Om \big(E;K;K\otimes (R\otimes _{H}K)\big),$$
it follows by Theorem \ref{thm:we=QE for dg-corings} that if $\Gamma\colon  H\to K$ is a homotopic relative Hopf algebra,  then $\Gal (\Gamma, \Gamma)$ is a relative weak equivalence, whence $\Gal (\Gamma, \vp)=\Om \big(E;K; \Gal (\Gamma, \Gamma)\big)$ is a relative weak equivalence and therefore that the second adjunction in Definition \ref{definition:hhge} is a Quillen equivalence, again by  Theorem \ref{thm:we=QE for dg-corings}.  We can thus conclude that $(\Gamma, \vp)$ is indeed a homotopic relative Hopf-Galois extension.

Since $K$ is relative cofibrant as a left $H$-module, $B$ is relative cofibrant as a left $A$-module.  To prove this, observe that the cellularly $r$-split filtration of $K$  as a left $H$-module induces a cellularly $r$-split filtration of $B$ as an $A$-module.  On other hand, the splitting $H\to K\to  H$ as left $K$-comodules and left $H$-modules induces a splitting $A\to B \to A$ of left $A$-modules, i.e., $B$  contains $A$ as a summand.

It follows that $\vp$ satisfies effective homotopic descent by Theorem \ref{thm:homotopic descent for chains}, since an argument similar to that in the proof of Proposition \ref{prop:hyp1} shows that $\vp$ is a chain homotopy equivalence because $\Gamma$ is. 
\end{proof}

Applying the homotopic normal basis construction, we establish a relative analogue of \cite[Proposition 12.1.8]{rognes} in the differential graded context.

\begin{proposition}\label{prop:rognes-analogue}  Let $H$ and $K$ be degreewise $R$-flat bialgebras in $\Ch_{R}$ such that $H$ is split-conilpotent. Let $\Gamma\colon   H\to K$ be a homotopic relative Hopf algebra  in $\Ch_R$ such that 
\begin{itemize}
\item as a left $H$-module, $K$ admits a cellularly $r$-split filtration,
\item $\Gamma$ admits a retraction $K\to H$ that is a morphism of $H$-modules and $K$-comodules, and
\item $\Cof(\Gamma)$ is degreewise $R$-flat.
\end{itemize}   
Let $(\Gamma, \vp)\colon  (H,A) \to (K,B)$ be a morphism in $\Comodalg_{\Ch_{R}}$ such that $B$ is degreewise $R$-flat.

If  $\vp^{\mathrm{hco}\Gamma}\colon A^{\mathrm{hco}H}\to B^{\mathrm{hco}K}$ is a relative weak equivalence, then
$(\Gamma, \vp)$ is a homotopic relative Hopf-Galois extension if and only if $\vp$ satisfies effective homotopic descent.

In particular, for any degreewise $R$-flat Hopf algebra $K$ and any morphism $(\eta, \vp)\colon  (R,A) \to (K,B)$ in $\Comodalg_{\Ch_{R}}$ such that $\vp^{\mathrm{hco}K}\colon A^{\mathrm{hco}K}\to B^{\mathrm{hco}K}$ is a relative weak equivalence, and $B$ is degreewise $R$-flat, $(\eta, \vp)$ is a homotopic relative Hopf-Galois extension if and only if $\vp$ satisfies effective homotopic descent.
\end{proposition}

\begin{proof} Our strategy in this proof is to exploit a comparison of $(\Gamma, \vp)$ with a homotopic normal basis extension. Note first that since  the coring $(R, K)$ is flat, by Proposition \ref{prop:R_X dualizable}, the functor $\Gamma^{*}\colon  \Alg^{K}\to \Alg^{H}$ is isomorphic to the cotensor product functor $-\cotensor_{K}H$.

Let $\Om(A;H;H)$ and $\Om(B;K;K)$ be the fibrant replacements of $A$ in $\Alg_R^{H}$ and of $B$ in $\Alg_R^{K}$ given by Corollary \ref{cor:hcoH-model}. Recall formula (\ref{eqn:phi-hcogamma}) for $\vp^{\mathrm{hco}\Gamma}$. Since by hypothesis $\vp^{\mathrm{hco}\Gamma}$ is a weak equivalence, and $H$ is split-conilpotent, Proposition \ref{prop:hyp1}  implies that the composite
$$\Om (A;H;H)\to \Om (\Gamma_{*}A;K;K)\cotensor_{K}{H}\to \Om (B;K;K)\cotensor_{K}{H}\cong \Om(B;K;H)$$
is also a relative weak equivalence.  Precomposing with  $\tilde\rho_{A}\colon  A\xrightarrow\sim  \Om (A;H;H)$, we obtain a relative weak equivalence of $H$-comodule algebras 
$$\alpha\colon  A \xrightarrow \sim \Om(B;K;H).$$

Set $A'=\Om (B;K;H)$, $B'=\Om (B;K;K)$, and 
$$\vp'=\Om (B;K;\Gamma)\colon    \Gamma_{*}A'\to B'.$$    
Proposition \ref{prop:hyp2} implies that $(\Gamma, \vp')\colon  (H,A')\to (K, B')$ is itself a homotopic relative Hopf-Galois extension and that $\vp'$ satisfies effective homotopic descent.

By Remark \ref{rmk:Galois functoriality},  the commuting diagram of comodule algebra morphisms
$$\xymatrix{(H,A) \ar[d]_{(\Id_{H}, \alpha)}\ar[r]^{(\Gamma, \vp)}&(K,B)\ar[d]^{(\Id_{K}, \tilde \rho_{B})}\\
(H,A') \ar[r]^{(\Gamma, \vp')}&(K,B')}$$
gives rise to a commuting diagram of functors 
$$\xymatrix{(\Ch_R) _{A}^{}\ar @<0.5ex>[rrr]^{\alpha_{*}}\ar @<-0.5ex>[dd]_{\Can_{\vp}}&&&(\Ch_R) _{A'}^{}\ar @<0.5ex>[lll]^{\alpha^{*}}\ar @<-0.5ex>[dd]_{\Can_{\vp'}}\\
\\
(\Ch_R)_{B}^{\Desc(\vp)}\ar @<-0.5ex>[uu]_{\Prim_{\vp}}\ar@<0.5ex>[rrr]^{\big(\tilde\rho_{B},\Desc (\alpha, \tilde\rho_{B})\big)_{*}}\ar @<-0.5ex>[dd]_{\Gal(\Gamma, \vp)_{*}}&&&(\Ch_R) _{B'}^{\Desc(\vp')}\ar @<0.5ex>[lll]^{\big(\tilde\rho_{B},\Desc (\alpha, \tilde\rho_{B})\big)^{*}}\ar @<-0.5ex>[dd]_{\Gal(\Gamma, \vp')_{*}}\ar @<-0.5ex>[uu]_{\Prim_{\vp'}}\\
\\
(\Ch_R) _{B}^{\Hopf(\rho_{B},\bar\mu_{K})}\ar @<-0.5ex>[uu]_{\Gal(\Gamma, \vp)^{*}}\ar @<0.5ex>[rrr]^{(\tilde\rho_{B}, \theta_{\Id_{K}, \tilde\rho_{B}})_{*}}&&&(\Ch_R) _{B'}^{\Hopf(\rho_{B'},\bar\mu_{K})},\ar @<0.5ex>[lll]^{(\tilde \rho_{B}, \theta_{\Id_{K}, \tilde\rho_{B}})^{*}}\ar @<-0.5ex>[uu]_{\Gal(\Gamma, \vp')^{*}}
}$$
where the $B'$-bimodule map underlying $\theta_{\Id_{K}, \tilde\rho_{B}}$,
$$B'\otimes \Cof (\Gamma)\otimes _{B}B'\to B'\otimes \Cof (\Gamma),$$
is given by the right $B'$-action on $B'\otimes \Cof(\Gamma)$ (cf.~Definition \ref{defn:hopf}).  Note that $B'\otimes \Cof (\Gamma)\otimes _{B}B'\to B'$ is flat  and $B'\otimes \Cof(\Gamma)$ is flat-cofibrant as a left $B'$-module, since $\Cof(\Gamma)$ is degreewise $R$-flat. Moreover $\theta_{\Id_{K}, \tilde\rho_{B}}$ is a  relative weak equivalence, since the canonical isomorphism 
$$B'\otimes \Cof (\Gamma)\otimes _{B}B\xrightarrow \cong B'\otimes \Cof (\Gamma)$$ 
factors as
$$B'\otimes \Cof (\Gamma)\otimes _{B}B\xrightarrow{B'\otimes \Cof(\Gamma) \otimes \tilde\rho_{B}} B'\otimes \Cof (\Gamma)\otimes _{B}B'\xrightarrow {\theta_{\Id_{K}, \tilde\rho_{B}}}  B'\otimes \Cof (\Gamma),$$
where the first map is a relative weak equivalence by Remark \ref{rmk:coinv} because it is equal to 
$$\tilde \rho_{B'\otimes \Cof (\Gamma)}\colon  B'\otimes \Cof (\Gamma) \to \Om \big( B'\otimes \Cof (\Gamma);K;K\big).$$
It follows that $\theta_{\Id_{K}, \tilde\rho_{B}}$ is a copure weak equivalence by  Theorem \ref{thm:copure}. 

Proposition \ref{prop:resext-chain} implies that the top adjunction is a Quillen equivalence, since $\alpha$ is a relative weak equivalence. The two vertical adjunctions on the right side of the diagram are Quillen equivalences, by Proposition \ref{prop:hyp2}.  Since $B'$ is homotopy faithfully flat as a $B$-module by Proposition \ref{prop:om(a,h,h)}, and $\theta_{\Id_{K}, \tilde\rho_{B}}$ is a copure weak equivalence, Corollary \ref{cor:QE conditions} implies that the bottom adjunction is a Quillen equivalence as well.

A  ``two-out-of-three'' argument enables us to conclude that $(\Gamma, \vp)\colon  (H,A) \to (K,B)$ is a relative homotopic Hopf-Galois extension if and only if $\vp\colon  A\to B$ satisfies effective homotopic descent.
\end{proof}

\begin{remark} We believe that it should be possible to generalize the strategy in the proof above to many other monoidal model categories, establishing an equivalence between homotopic Hopf-Galois extensions and morphisms satisfying effective homotopic descent when the induced map on the coinvariants is a weak equivalence.   The key to the proof is the existence of a well-behaved construction, replacing any (nice enough) morphism of comodule algebras by a weakly equivalent morphism of comodule algebras that is a homotopic Hopf-Galois extension and that satisfies effective homotopic descent.  A ``homotopic normal extension'' of the sort employed in the proof above should do the trick in monoidal model categories with compatible simplicial structure.
\end{remark}

The close relationship between homotopic Hopf-Galois extensions and morphisms satisfying effective homotopic descent enables us to include the notion of Koszul duality in our general picture as well. 

\begin{proposition}\label{prop:koszul} Let $(\Gamma, \vp)\colon  (H,A) \to (K,B)$ be a relative homotopic Hopf-Galois extension in $\Comodalg_{\Ch_R}$  such that $\vp\colon A \to B$ satisfies effective homotopic descent, $\Cof(\Gamma)$  and $B$ are degreewise $R$-flat, and $B$ is augmented.  

If the unit map $\eta:R \to B$ is a chain homotopy equivalence, then $\Cof (\Gamma)$ is a \emph{generalized Koszul dual} of $A$, in the sense that homotopy category of right $A$-modules is equivalent to the homotopy category of right $\Cof(\Gamma)$-comodules.
\end{proposition} 

\begin{proof} Since $(\Gamma, \vp)\colon  (H,A) \to (K,B)$ is a relative homotopic Hopf-Galois extension, and $\vp\colon A \to B$ satisfies effective homotopic descent, there is a chain of Quillen equivalences
$$\xymatrix{  (\Ch_R) _{A} \ar@<1ex>[rr]^-{\Can_{\vp}} && (\Ch_R)_{B}^{\Desc(\vp)} \ar@<1ex>[ll]^-{\Desc(\vp)} \ar@<1ex>[rr]^-{\Gal(\Gamma, \vp)_{*}} && (\Ch_R) _{B}^{\Hopf(\rho_{B},\bar\mu_{K})}. \ar@<1ex>[ll]^-{\Gal(\Gamma, \vp)^{*}}}$$
The unit map $\eta:R \to B$ induces a morphism in $\Pair_{\VV}$
$$(\Id_{K}, \eta, \Id_{\Cof(\Gamma)}): \big(K,R, \Cof (\Gamma)\big) \to \big(K,B, \Cof (\Gamma)\big)$$
and therefore a morphism of corings with underlying morphism of chain complexes
$$\Id_{B} \otimes \bar \mu _{K}: B\otimes \Cof (\Gamma)\otimes B \to B\otimes \Cof (\Gamma),$$
which is a relative weak equivalence, since $\eta$ is a chain homotopy equivalence, and $\bar \mu_{K}(\Id_{\Cof(\Gamma)}\otimes \eta)=\Id_{\Cof(\Gamma)}$.   Since $\Cof(\Gamma)$ is degreewise $R$-flat, the source of $\Id_{B} \otimes \bar \mu _{K}$ is flat and its target  flat-cofibrant as a left $B$-module.  It follows that $\Id_{B} \otimes \bar \mu _{K}$ is a copure weak equivalence of $B$-corings.   Moreover, $B$  homotopy faithfully flat as an $R$-module.  It is strongly homotopy flat over $R$, since degreewise $R$-flat by hypothesis, and all $R$-modules are homotopy flat, and it is homotopy faithful as a $R$-module, since it is augmented.  Corollary \ref{cor:QE conditions} therefore implies that
$$\bigadjunction{(\Ch_R)_{R}^{\Cof(\Gamma)}}{(\Ch_R)_{B}^{\Hopf(\rho_{B},\bar\mu_{K})}}{}{}$$
is a Quillen equivalence, whence
$$\Ho \big((\Ch_R) _{A}\big)\simeq \Ho \big( (\Ch_R)^{\Cof(\Gamma)}\big)$$
as desired.
\end{proof}

The connection between Koszul duality and homotopic Hopf-Galois extensions hinted at here will be explored further in a forthcoming paper.
    
\subsection{Principal fibrations and Hopf-Galois extensions}

We show in this section that simplicial principal fibrations naturally give rise to homotopic Hopf-Galois extensions in $\Ch_{R}$, confirming that our definition is reasonable.   The example elaborated here generalizes \cite[Example 4.4.7]{karpova}.

In this section, for any simplicial set $X$, $C_{*}X$ denotes the normalized chains on $X$ with coefficients in $R$, which admits a natural coalgebra structure.  For any reduced simplicial set $X$, we let $\G  X$ denote the Kan loop group construction on $X$ \cite[\S 27]{may}.   Recall that the geometric realization $|\G  X|$ has the homotopy type of $\Om |X|$, the based loop space on the geometric realization of $|X|$.

We recall from \cite{hpst} that if $X$ is a simplicial set that is $1$-reduced (i.e., $X_{0}$ and $X_{1}$ are both singletons), then the differential graded algebra $\Om C_{*}X$ obtained by applying the reduced cobar construction to $C_{*}X$ admits a natural comultiplication 
$$\psi_{X}:\Om C_{*}X\to \Om C_{*}X\otimes \Om C_{*}X$$ 
endowing $\Om C_{*}X$ with the structure of a bialgebra in $\Ch_{R}$.  This comultiplicative structure is topologically meaningful, in the sense that the natural morphism of differential graded algebras 
$$\alpha_{X}:\Om C_{*}X \to C_{*}\G  X$$
first defined by Szczarba \cite{szczarba} is naturally strongly homotopy comultiplicative with respect to $\psi_{X}$ and to the usual comultiplication on $C_*(\G  X)$, i.e., it gives rise to a natural morphism of differential graded algebras
$$\beta_{X}:\Om^{2}C_{*}X \to \Om C_{*}\G  X.$$
Moreover, as shown in \cite{hess-tonks}, the chain map underlying $\alpha_{X}$ admits a natural chain homotopy inverse.  The existence of both the higher homotopies for the strongly homotopy comultiplicative structure of $\alpha _{X}$ and the chain inverse to $\alpha_{X}$ follow from acyclic models arguments, whence $\beta_{X}$ must also admit a natural chain homotopy inverse.

Let $X$ be a $2$-reduced simplicial set, and $Y$ a $1$-reduced simplicial set equipped with a twisting function $\tau: Y \to \G  X$.    Let $j: \G X \to \G  X\times_{\tau}Y$ denote the inclusion of $\G X$ into the twisted cartesian product of $\G X$ and $Y$ determined by the twisting function $\tau$ \cite[\S 18]{may}.  Let 
$$\vp=\Om C_{*}j: \Om C_{*}\G X \to \Om C_{*}(\G X\times_{\tau}Y),$$
which is a morphism of bialgebras. Note that the projection map $\G X\times_{\tau}Y\to Y$ gives rise to the structure of a $\Om C_{*}Y$-comodule algebra on $\Om C_{*}(\G  X\times_{\tau }Y)$.

\begin{proposition}\label{prop:prin-fib}  If $X$ is a simplicial double suspension, and $Y$ is a simplicial suspension \cite [\S 27]{may}, then for every simplicial map $g:Y\to X$, the morphism of comodule algebras
$$(R, \Om C_{*}\G  X) \to \big(\Om C_{*}Y, \Om C_{*}(\G X\times_{\tau}Y)\big),$$
where $\eta: R\to \Om C_{*}Y$ denotes the unit map and $\tau=\tau_{X}g$ with $\tau_{X}:X\to \G X$ the universal twisting function, is a homotopic Hopf-Galois extension.
\end{proposition}

We believe that this proposition holds even when $X$ and $Y$ are not suspensions, but developing the general argument would require too great a digression from the theme of this article to be reasonably presented here.

\begin{proof}  According to Proposition \ref{prop:hhg conditions}, it suffices to show that 
$$\vp^{\mathrm{hco} \eta}\colon   \Om C_{*}\G X=(\Om C_{*}\G X)^{\mathrm{hco} R} \to \Om C_{*}(\G X\times_{\tau}Y)^{\mathrm{hco} \Om C_{*}Y}$$
is a relative equivalence of algebras and that 
$$\Gal(\eta, \vp): \Om C_{*}(\G X\times _{\tau}Y)\otimes _{\Om C_{*}\G  X}\Om C_{*}(\G X\times_{\tau}Y) \to \Om C_{*}(\G X\times _{\tau}Y)\otimes \Om C_{*}Y$$
is a copure weak equivalence of corings.

We treat first the case of $\vp^{\mathrm{hco} \eta}$, for which our proof does not actually need $X$ and $Y$ to be simplicial suspensions.  By Corollary \ref{cor:fib-repl}, we can take 
$$\Om C_{*}(\G X\times_{\tau_{Y} }Y)^{\mathrm{hco} \Om C_{*}Y}=\Om \big( \Om C_{*}(\G X\times_{\tau }Y); \Om C_{*}Y; R),$$
since $\Om C_{*}Y$ is degreewise $R$-free and therefore degreewise $R$-flat.  A model of $\vp^{\mathrm{hco} \eta}$ is then given by 
\begin{equation}\label{eqn:vp-hco}
\Om (\vp; R;R): \Om C_{*}\G X=\Om (\Om C_{*}\G X;R;R) \to \Om \big( \Om C_{*}(\G X\times_{\tau }Y); \Om C_{*}Y; R).
\end{equation}
On the other hand, there is a sequence of relative weak equivalences of algebras in $\Ch_{R}$
$$\xymatrix{\Om \big( \Om C_{*}(\G X\times_{\tau}Y); \Om C_{*}Y; R)\ar[rr]^{\Om (\alpha; \alpha; R)}_{\simeq}&& \Om \big( C_{*}\G (\G X\times_{\tau}Y); C_{*}\G Y; R)\ar [d]_{\simeq}\\
&& C_{*}\big( \G (\G X\times_{\tau}Y) \times_{\tau'} \G ^{2}Y\big)\ar [d]_{\simeq}\\
&&C_{*}\G ^{2} X,}$$
where $\tau':  \G (\G X\times_{\tau}Y)\to \G ^{2}Y$ is the twisting function given by projection onto $\G  Y$, followed by the canonical twisting function $\tau_{\G Y}:  \G Y \to \G ^{2}Y$.  The fact that the first vertical arrow, which extends $\alpha_{\G X\times_{\tau }Y}$, is a relative weak equivalence follows from \cite[Theorem 2.23]{farjoun-hess}, which is a straightforward generalization of the main theorem in \cite{hess-tonks}.  The second vertical relative weak equivalence is a consequence of the fact that the inclusion 
$$\G ^{2}X \hookrightarrow\G (\G X\times_{\tau}Y) \times_{\tau'} \G ^{2}Y$$ admits a retraction that is a homotopy inverse (cf.~Lemma \ref{lem:simparg}).  Precomposing this sequence with (\ref{eqn:vp-hco}) gives exactly $\alpha_{X}$, which is also a relative weak equivalence, whence, by two-out-of-three, (\ref{eqn:vp-hco}) is a relative weak equivalence as well, as desired.

Concerning $\Gal(\eta, \vp)$, we begin by recalling from \cite[Proposition 4.3.11, Remark 4.3.12]{karpova} that since the nondifferential algebra map $\vp$ is an  inclusion into a free extension, $\Om C_{*}(\G X\times_{\tau}Y)$ admits a cellularly $r$-split filtration as a left (or right) $\Om C_{*}\G X$-module, whence it is homotopy flat, by Propositions \ref{prop:bmr-cofib} and \ref{prop:ch-special-modules}(a).  Since normalized chain complexes are $R$-free, we can choose this filtration so that each filtration quotient is $\Om C_{*}\G X$-free on an $R$-free module.  It follows that $$\Om C_{*}(\G  X\times _{\tau }Y)\otimes _{\Om C_{*}\G X}\Om C_{*}(\G X\times_{\tau}Y)$$ 
is flat as a left $\Om C_{*}(\G X\times_{\tau}Y)$-module.  Moreover, since the differential on $\Om C_{*}(\G X\times _{\tau }Y)\otimes \Om C_{*}Y$ is the usual tensor differential, it is certainly flat-cofibrant as a left $\Om C_{*}(\G X\times_{\tau}Y)$-module, as well as coaugmented.  By Theorem \ref{thm:copure}, it suffices therefore to show that  $\Gal(\eta, \vp)$ is a relative weak equivalence.

In \cite{hps} the authors show that if a simplicial set $Z$ is the simplicial suspension of another simplicial set $W$, then the natural coalgebra structure on the normalized chain complex $C_{*}Z$ is trival, and that, when endowed with its canonical multiplication, the cobar construction $\Om C_{*}Z$ is isomorphic to the tensor algebra $TC_{*}W$, endowed with the comultiplication induced by that on $C_{*}W$.  Moreover the Szczarba equivalence $\alpha_{Z}:\Om C_{*}Z \to C_{*}\G Z$ respects the comultiplication strictly.  If $W$ itself is a simplicial suspension, then the generators of the bialgebra $\Om C_{*}Z$ are actually primitive and thus $\Om C_{*}Z$ is cocommutative.

Since $\Om (R; C_{*}Y; C_{*}Y)$ is a relative cofibrant left $\Om C_{*}Y$-module, the Szczarba equivalence $\alpha_{X}:\Om C_{*}X\to C_{*}\G X$ induces a relative weak equivalence of left $\Om C_{*}X$-modules

{\small\begin{equation}\label{eqn:sz1}\Om (R; C_{*}X; C_{*}Y)\cong\Om C_{*}X \otimes _{\Om C_{*}Y} \Om (R; C_{*}Y; C_{*}Y) \xrightarrow {\simeq} C_{*}\G X \otimes_{\Om C_{*}Y}  \Om (R; C_{*}Y; C_{*}Y),\end{equation}}

\noindent where we regard $\Om C_{*}X$ as a right $\Om C_{*}Y$-module via $\Om C_{*}g: \Om C_{*}Y \to \Om C_{*}X$.
On the other hand, Szczarba proved in \cite[Theorems 2.2-2.4]{szczarba} that there is a quasi-isomorphism 
\begin{equation}\label{eqn:sz2}C_{*}\G X \otimes_{\Om C_{*}Y}  \Om (R; C_{*}Y; C_{*}Y)\to C_{*}(\G X\times_{\tau }Y)
\end{equation}
of left $C_{*}\G X$-modules, extending the equivalence $\Om C_{*}X\to C_{*}\G X$ by the identity on the $\Om C_{*}Y$-component. A straightforward generalization of the main theorem in \cite{hess-tonks} shows that this quasi-isomorphism admits a chain homotopy inverse.

As the generators of $\Om C_{*}X$ are primitive, and the differential on $\Om (R; C_{*}X; C_{*}Y)$ sends elements of $C_{*}Y$ to generators of $\Om C_{*}X$, it is easy to show that $\Om (R; C_{*}X; C_{*}Y)$ admits a cocommutative comultiplication extending those on $\Om C_{*}X$ and $C_{*}Y$, with no perturbation.  Moreover, the composite of equivalences (\ref{eqn:sz1}) and (\ref{eqn:sz2}) is a strict map of coalgebras with respect to this comultiplication, since it  extends $\Om C_{*}X\to C_{*}\G X$ simply by the identity on the $C_{*}\Om Y$-component.  Applying the cobar construction to the composite equivalence of coalgebras given by (\ref{eqn:sz1}) followed by (\ref{eqn:sz2}), we obtain a quasi-isomorphism of chain algebras
$$\Om^{2} (R; C_{*}X; C_{*}Y)\xrightarrow{\simeq} \Om C_{*}(\G X\times_{\tau }Y),$$
which admits a chain homotopy inverse as left $\Om ^{2}C_{*}X$-modules, since  the source and target are bifibrant with respect to the projective model category structures on ${}_{\Om ^{2}C_{*}X}(\Ch_{R})$.  Note that we use here that the target is a quasi-free extension of the source, as chain algebras. 

Consider $\Om^{2} C_{*}X$ as a right $\Om^{2} C_{*}Y$-module via $\Om^{2} C_{*}g: \Om^{2} C_{*}Y \to \Om^{2} C_{*}X$. It follows from \cite [Corollary 3.6]{hess-levi} and \cite[Lemma 4.3.21]{karpova} that 
$$\Om (R; \Om C_{*}X; \Om C_{*}Y)= \Om^{2} C_{*}X\otimes_{\Om ^{2}C_{*}Y}\Om (R; \Om C_{*}Y; \Om C_{*}Y)$$
admits a natural chain bialgebra structure such that there is a quasi-isomorphism of chain algebras
$$\Om^{2} (R; C_{*}X; C_{*}Y)\xrightarrow\simeq \Om (R; \Om C_{*}X; \Om C_{*}Y).$$
As above, since the source and target are bifibrant with respect to the projective model category structures on ${}_{\Om ^{2}C_{*}X}(\Ch_{R})$, this quasi-isomorphism admits a chain homotopy inverse as left $\Om ^{2}C_{*}X$-modules.  We therefore have a zigzag 
$$ \Om (R; \Om C_{*}X; \Om C_{*}Y) \rightleftarrows \Om^{2} (R; C_{*}X; C_{*}Y) \leftrightarrows  \Om C_{*}(\G X\times_{\tau }Y)$$
of relative weak equivalences of $\Om^{2} C_{*}Y$-modules.  Moreover, both of the outward-pointing maps are easily seen to be morphism of $\Om C_{*}Y$-comodules.    Because all three objects are cofibrant $\Om^{2}C_{*}X$-modules, and $\Om C_{*}(\G X\times_{\tau }Y)$ is a cofibrant $\Om C_{*}\G X$-module, there is a zigzag of relative equivalences of $\Om C_{*}Y$-comodules
$$\xymatrix{\Om (R; \Om C_{*}X; \Om C_{*}Y) \otimes_{\Om^{2} C_{*}X}\Om (R; \Om C_{*}X; \Om C_{*}Y)\\ \Om^{2} (R; C_{*}X; C_{*}Y)\otimes_{\Om^{2} C_{*}X}\Om^{2} (R; C_{*}X; C_{*}Y)\ar [u]_{\simeq}\ar [d]_{\simeq}\\
\Om C_{*}(\G X \times_{\tau } Y) \otimes_{\Om C_{*}\G X}\Om C_{*}(\G X \times_{\tau } Y).}$$

To show that $\Gal (\eta, \vp)$ is a relative weak equivalence, it suffices therefore to prove that the composite
\begin{equation}\label{eqn:alt-gal}\xymatrix{\Om (R; \Om C_{*}X; \Om C_{*}Y) \otimes_{\Om^{2} C_{*}X}\Om (R; \Om C_{*}X; \Om C_{*}Y)\ar [d]\\
\Om (R; \Om C_{*}X; \Om C_{*}Y) \otimes_{\Om^{2} C_{*}X}\Om (R; \Om C_{*}X; \Om C_{*}Y)\otimes \Om C_{*}Y\ar[d]\\
\Om (R; \Om C_{*}X; \Om C_{*}Y)\otimes \Om C_{*}Y}
\end{equation}
is a relative weak equivalence, where the first map is given by the $\Om C_{*}Y$-coaction and the second by the multiplication in $\Om (R; \Om C_{*}X; \Om C_{*}Y)$.  

Observe that, if we ignore differentials,
$$\Om (R; \Om C_{*}X; \Om C_{*}Y) \otimes_{\Om^{2} C_{*}X}\Om (R; \Om C_{*}X; \Om C_{*}Y)\cong \Om^{2} C_{*}X\otimes \Om C_{*}Y\otimes \Om C_{*}Y.$$
By Koppinen's Lemma
\cite[Lemma 4.4.1]{schauenburg} (again ignoring differentials), the composite (\ref{eqn:alt-gal}) is an isomorphism of $\Om (R;\Om C_{*}X; \Om C_{*}Y)$-modules and $\Om C_{*}Y$-comodules   if and only if the map
$$\eta\otimes 1:\Om C_{*}Y\to \Om^{2} C_{*}X\otimes \Om C_{*}Y$$
is invertible with respect to the convolution product on 
$$\Hom \big(\Om C_{*}Y,  \Om (R;\Om C_{*}X; \Om C_{*}Y)\big),$$ 
where $\Hom$ means the internal hom of graded $R$-modules.  Since $\Om C_{*}Y$ has an antipode $S$ (because it is a graded bialgebra), and $\eta\otimes 1$ is an algebra map (cf.~formulas in \cite[Corollary 3.6]{hess-levi}), it follows that $(\eta\otimes 1)S$ is a convolution inverse to $\eta\otimes 1$, and thus that the composite (\ref{eqn:alt-gal}) is an isomorphism, at least as graded modules.  Its inverse is necessarily also a chain map, however, because it is inverse to a chain map and therefore  (\ref{eqn:alt-gal}) is an isomorphism of chain complexes. We conclude that $\Gal(\eta, \vp)$ must be a relative weak equivalence, as desired.
\end{proof}

\begin{remark}\label{rmk:Koszul-prinfib} It follows from Propositions \ref{prop:koszul} and \ref{prop:prin-fib} that if $C_{*}(\G X\times_{\tau}Y)$ is contractible, then
$$\Ho \big((\Ch_R) _{\Om C_{*}\G X}\big)\simeq \Ho \big( (\Ch_R)^{\Om C_{*}Y}\big),$$ 
at least under the additional hypothese on $X$ and $Y$ in the statement of the proposition.
When applied to the universal bundle $\G X\times_{\tau _{X}}X$,  this equivalence becomes
$$\Ho \big((\Ch_R) _{\Om C_{*}\G X}\big)\simeq \Ho \big( (\Ch_R)^{\Om C_{*}X}\big),$$ 
at least when $X$ is a double suspension.
\end{remark}

\appendix

\section{A technical lemma}

In this section we prove a technical lemmas used in the proof of Proposition \ref{prop:prin-fib}

\begin{lemma}\label{lem:simparg} For any $1$-reduced simplicial set $X$ and $1$-reduced simplicial set $Y$ equipped with twisting function $\tau: Y\to \G  X$, the inclusion 
$$\G ^{2}X \hookrightarrow\G (\G X\times_{\tau_{Y}}Y) \times_{\tau'} \G ^{2}Y$$ is a simplicial homotopy equivalence, where $\tau':  \G (\G X\times_{\tau}Y)\to \G ^{2}Y$ is the twisting function given by projection onto $\G  Y$, followed by the canonical twisting function $\tau_{\G Y}:  \G Y \to \G ^{2}Y$. 
\end{lemma}

\begin{proof} Note first that if $\vp: G \to H$ is a surjective simplicial homomorphism, then $\ker \vp$ acts principally on $G$, i.e., $\vp$ is a principal $\ker \vp$-fibration, and therefore a Kan fibration \cite[Lemma 18.2]{may}. Moreover, the inclusions of $\G ^{2}X$ into $\G ( \G  X \times_{\tau _{X}} X)$ and $\G ^{2}X\times_{\tau_{\G X}} \G  X$ are both cofibrations in the Kan model structure on $\mathsf{sSet}$.  The sequences
$$\G ^{2}X\to \G ( \G  X \times_{\tau _{X}} X)\to \G X \quad\text{and}\quad \G ^{2}X\to \G ^{2}X\times_{\tau_{\G X}} \G  X\to \G Y$$
are therefore bifibrant objects in the slice category $\G ^{2}X/\mathsf {sSet}/\G  X$ of simplicial sets over $\G X$ and under $\G ^{2}X$.  Since both objects are weakly equivalent in  $\G ^{2}X/\mathsf {sSet}/\G  X$ to the sequence $\G ^{2}X\to * \to \G X$, there are homotopy equivalences
$$\G ( \G  X \times_{\tau _{X}} X)\rightleftarrows  \G ^{2}X\times_{\tau_{\G X}}\G X
$$
in $\G ^{2}X/\mathsf {sSet}/\G  X$.  Pulling back over the homomorphism $\alpha_{\tau}:\G  Y \to \G  X$ induced by the twisting function $\tau$ gives rise to homotopy equivalences
\begin{equation}\label{eqn:he'} \G ( \G  X \times_{\tau _{X}} X)\times _{\G X}\G  Y\rightleftarrows  \G ^{2}X\times_{\tilde\tau}\G Y\end{equation}
in $\G ^{2}X/\mathsf {sSet}/\G  Y$, where $\tilde\tau=\tau_{\G X}\alpha_{\tau}:\G Y \to \G ^{2}X$ .

For the next step in our argument, it is important to observe that for any pair of simplicial maps with common target
$$V\to Z \leftarrow W,$$
the natural simplicial map
$$\pi: \G (V\times_{Z}W) \to \G V \times_{\G  Z}\G  W$$
admits a homotopy inverse in the slice category $\G V/\mathsf{sSet}/\G W$ of simplicial sets under $\G V$ and over $\G W$.  Indeed, the projections from both objects to $\G  W$ are Kan fibrations in $\mathsf{sSet}$, and  the obvious maps 
$$\G V\to \G (V\times_{Z}W)\quad\text{ and }\quad \G V\to\G V \times_{\G  Z}\G  W$$ are inclusions and therefore cofibrations in the Kan model category structure on $\mathsf{sSet}$, so that 
$$\G V \to \G (V\times_{Z}W)\to \G  W \quad \text{and} \quad \G V \to \G V \times_{\G  Z}\G  W\to \G  W$$
are bifibrant objects in $\G V/\mathsf{sSet}/\G W$. Furthermore, $\pi$ is a weak equivalence, since the geometric realization of $\pi$ is simply the homotopy equivalence $\Om\big( |V| \times _{|Z|}|W|\big) \to \Om |V| \times _{\Om [Z]} \Om |W|$.  Here we use that $|\G  K| \simeq \Om |K|$ for any reduced simplicial set $K$ and that geometric realization commutes with pullbacks, as can be shown easily \cite{nlab:geom-real}.  It follows that $\pi$ is in fact a homotopy equivalence in $\G V/\mathsf{sSet}/\G W$, as desired.

As a special case of the result above, we obtain homotopy equivalences
\begin{equation}\label{eqn:he''}\G (\G  X \times_{\tau} Y)=\G \big ( (\G  X \times_{\tau_{X}}X)\times_{X} Y\big) \rightleftarrows\G  (\G  X \times_{\tau_{X}}X)\times_{\G X} \G Y
\end{equation}
 in $\G ^{2}X/\mathsf{sSet}/\G Y$.  
 
 Combining (\ref{eqn:he'}) and (\ref{eqn:he''}) gives rise to homotopy equivalences.
$$\G (\G  X \times_{\tau} Y) \rightleftarrows \G ^{2}X\times_{\tilde\tau}\G Y$$
 in $\G ^{2}X/\mathsf{sSet}/\G Y$ and therefore to homotopy equivalences
 $$\G (\G  X \times_{\tau} Y)\times_{\tau '}\G ^{2} Y \rightleftarrows \G ^{2}X\times_{\tilde\tau}\G Y\times_{\tau_{\G Y}}\G ^{2} Y$$
 in $\G ^{2}X/\mathsf{sSet}$.  Since $\G ^{2}X$ is a strong deformation retract of  $\G ^{2}X\times_{\tilde\tau}\G Y\times_{\tau_{\G Y}}\G ^{2} Y$, we can conclude.
\end{proof}

 \bibliographystyle{amsplain}
\bibliography{vstructures2}
\end{document}